\documentclass{amsart}

\usepackage{amsmath}
\usepackage{amssymb}
\usepackage{amsthm}
\usepackage{mathscinet}
\usepackage{graphicx}
\usepackage{enumitem}

\theoremstyle{plain}
\newtheorem{thm}{Theorem}[section]
\newtheorem{prop}[thm]{Proposition}
\newtheorem{lem}[thm]{Lemma}
\newtheorem{cor}[thm]{Corollary}

\theoremstyle{definition}
\newtheorem{defn}[thm]{Definition}
\newtheorem{claim}{Claim}[thm]

\newtheorem{question}{Question}
\newtheorem{example}{Example}

\begin{document}

\title[Radial departures]{Radial departures and plane embeddings\\of arc-like continua}
\author{Andrea Ammerlaan \and Ana Anu\v{s}i\'{c} \and Logan C. Hoehn}
\date{\today}

\address{Nipissing University, Department of Computer Science \& Mathematics, 100 College Drive, Box 5002, North Bay, Ontario, Canada, P1B 8L7}
\email{ajammerlaan879@my.nipissingu.ca}
\email{anaa@nipissingu.ca}
\email{loganh@nipissingu.ca}

\dedicatory{Dedicated to the memory of Piotr Minc}

\thanks{This work was supported by NSERC grant RGPIN-2019-05998}


\begin{abstract}
We study the problem of Nadler and Quinn from 1972, which asks whether, given an arc-like continuum $X$ and a point $x \in X$, there exists an embedding of $X$ in $\mathbb{R}^2$ for which $x$ is an accessible point.  We develop the notion of a radial departure of a map $f \colon [-1,1] \to [-1,1]$, and establish a simple criterion in terms of the bonding maps in an inverse system on intervals to show that there is an embedding of the inverse limit for which a given point is accessible.  Using this criterion, we give a partial affirmative answer to the problem of Nadler and Quinn, under some technical assumptions on the bonding maps of the inverse system.
\end{abstract}

\maketitle

\section{Introduction}
\label{sec:intro}

In this paper we study plane embeddings of arc-like continua and their accessible sets.  The question of understanding properties of planar embeddings of arc-like continua has generated significant interest over the years, see for example \cite{mazurkiewicz1929, brechner1978, lewis1981, mayer1982, mayer1983, minc-transue1992, debski-tymchatyn1993, minc1997, anusic-bruin-cinc2017, anderson-choquet1959, ozbolt2020}.  We are motivated by the question of Nadler and Quinn from 1972 (\cite[p.229]{nadler1972} and \cite{nadler-quinn1972}), which asks whether for every arc-like continuum $X$, and every $x \in X$, there is an embedding $\Omega$ of $X$ in the plane $\mathbb{R}^2$ such that $x$ is an \emph{accessible} point, i.e.\ such that there is an arc $A \subset \mathbb{R}^2$ with $A \cap \Omega(X) = \{\Omega(x)\}$ (for other definitions and notation, see the Preliminaries section below).  The Nadler-Quinn problem has appeared in various compilations of continuum theory problems, including as Problem~140 in \cite{lewis1983}, as Question~16 (by J.C.\ Mayer, for indecomposable continua) in \cite{problems2002}, and as Problem~1 (by H.\ Bruin, J.\ \v{C}in\v{c} and the second author), Problem~48 (by W.\ Lewis), and Problem~50 (by P.\ Minc, for indecomposable continua) in \cite{problems2018}, but to this day remains open.

An alternative formulation of the Nadler-Quinn question is as follows.  Given an arc-like continuum $X$ and a point $x \in X$, we can construct a (simple-triod-like) continuum $X_x = X \cup A$, where $A$ is an arc such that $A \cap X = \{x\}$.  Then the question of whether $X$ can be embedded in the plane so as to make $x$ accessible is equivalent to the question of whether this continuum $X_x$ can be embedded in the plane.  In this way, the Nadler-Quinn problem asks whether all spaces from this class of triod-like continua can be embedded in the plane.  Answering this would be a step forward towards understanding which tree-like continua can be embedded in the plane, which is a question of central importance in the field.

P.\ Minc identified a particular, simple example of an arc-like continuum $X_M = \varprojlim \left \langle [0,1],f_M \right \rangle$, where $f_M$ is the piecewise-linear map shown in Figure~\ref{fig:minc}, and asked (see \cite[Question~19, p.335]{problems2002} and \cite[p.297]{problems2018}) whether $X$ can be embedded in the plane so as to make the point $\langle \frac{1}{2},\frac{1}{2},\ldots \rangle \in X$ accessible.  This was recently answered affirmatively by the second author in \cite{anusic2021}, who proved more generally that for any simplicial, locally eventually onto map $f \colon [0,1] \to [0,1]$, and any point $x \in X = \varprojlim \left \langle [0,1],f \right \rangle$, there is an embedding of $X$ in the plane for which $x$ accessible.  In this paper we give another partial affirmative answer to the question of Nadler and Quinn, for a different class of inverse systems, where in particular the bonding maps are not assumed to be all the same map.

\begin{figure}
\begin{center}
\includegraphics{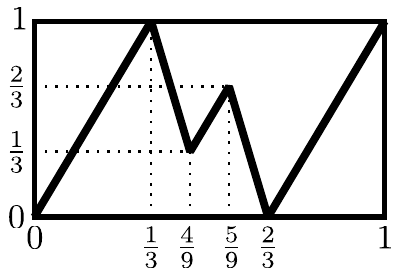}
\end{center}

\caption{The map $f_M$ defined by P.\ Minc.}
\label{fig:minc}
\end{figure}

With the use of the Anderson-Choquet Embedding Theorem \cite{anderson-choquet1959}, in \cite{anusic-bruin-cinc2020} Bruin, \v{C}in\v{c}, and the second author prove that given a point $x = \langle x_1,x_2,\ldots \rangle$ in an arc-like continuum $X = \varprojlim \left \langle [0,1],f_n \right \rangle$, there is an embedding of $X$ in the plane for which $x$ is accessible as long as, loosely speaking, $x_n$ is not in a ``zig-zag'' pattern in $f_n$.  The approach taken in \cite{anusic2021} is to find an alternative inverse limit representation of the arc-like continuum $X$ for which the coordinates of the point $x$ are not in such ``zig-zags''.  In this paper, we continue to explore this approach, and introduce a standard factorization of an interval map, the \emph{radial contour factorization}, as a means to find alternative inverse limit representations of arc-like continua.  We break the notion of a ``zig-zag'' into two more fundamental parts, namely \emph{radial departures} (of opposite orientations).

This paper is organized as follows.  After some preliminary definitions and notation, in Section~\ref{sec:prelim} we give a simple reduction of the question of Nadler and Quinn to the case of the point $x = \langle 0,0,\ldots \rangle$ in an inverse limit space $X = \varprojlim \left \langle [-1,1],f_n \right \rangle$, where $f_n(0) = 0$ for each $n$.  In Section~\ref{sec:contour} we introduce the (``one-sided'') notions of \emph{departures} and the \emph{contour factorization} of a map $f \colon [0,1] \to [-1,1]$.  We then introduce the notion of a \emph{radial departure} in Section~\ref{sec:rad deps}, and in Section~\ref{sec:embeddings} we show that there is an embedding of $X = \varprojlim \left \langle [-1,1],f_n \right \rangle$ in the plane for which $\langle 0,0,\ldots \rangle$ is accessible as long as for each $n$, $f_n$ does not contain both positive and negative radial departures.  In Section~\ref{sec:rad contour} we introduce the \emph{radial contour factorization}, which we use in Section~\ref{sec:no twins} to give a partial affirmative answer (Theorem~\ref{thm:no twins}) to the Nadler-Quinn problem under certain assumptions on the bonding maps $f_n$.

\section{Preliminaries}
\label{sec:prelim}

Throughout this paper, the word \emph{map} will always mean a continuous function.  A \emph{continuum} is a compact, connected metric space.

An \emph{inverse sequence} is a sequence of spaces and maps
\[ \langle X_n,f_n \rangle = \langle X_n,f_n \rangle_n = \langle X_1,f_1,X_2,f_2,X_3,\ldots \rangle \]
where for each $n$, $f_n$ is a map from $X_{n+1}$ to $X_n$.  The spaces $X_n$ are called \emph{factor spaces} and the maps $f_n$ are called \emph{bonding maps}.  If $S$ is an infinite subset of $\mathbb{N}$, then $\langle X_n,f_n \rangle_{n \in S}$ refers to the inverse system $\langle X_{n_k},f_{n_k} \rangle_k$, where $\{n_k: k \in \mathbb{N}\}$ is an increasing enumeration of $S$.  In this situation, it is understood that for each $n \in S$, $f_n$ is a map from $X_{n'}$ to $X_n$, where $n'$ is the smallest element of $S$ which is greater than $n$.

The \emph{inverse limit} of an inverse sequence is the set
\[ \varprojlim \langle X_n,f_n \rangle = \left\{ \langle x_n \rangle_{n=1}^\infty: f_n(x_{n+1}) = x_n \textrm{ for each } n\right\} \]
considered as a subspace of the product $\prod_{n=1}^\infty X_n$.  We record two standard properties of inverse limits:
\begin{itemize}
\item (Dropping finitely many coordinates) For any $n_0 \in \mathbb{N}$, $\varprojlim \langle X_n,f_n \rangle \approx \varprojlim \langle X_n,f_n \rangle_{n \geq n_0}$.
\item (Composing bonding maps) Let $S$ be an infinite subset of $\mathbb{N}$.  Given $n \in S$, let $n'$ be the smallest element of $S$ which is greater than $n$, and define $f_n^\circ = f_n \circ f_{n+1} \circ \cdots \circ f_{n'-1}$, which is a map from $X_{n'}$ to $X_n$.  Then $\varprojlim \langle X_n,f_n \rangle \approx \varprojlim \langle X_n,f_n^\circ \rangle_{n \in S}$.
\end{itemize}

A map $f \colon [-1,1] \to [-1,1]$ is \emph{piecewise-linear} if there is a finite set $S \subset [-1,1]$ with $-1,1 \in S$ such that on the intervals between points of $S$, $f$ is linear.  An inverse system $\langle [-1,1],f_n \rangle$ is \emph{simplicial} if there exist finite sets $S_1,S_2,\ldots$ such that for each $n$, $-1,1 \in S_n$, $f_n(S_{n+1}) \subseteq S_n$, and if $I$ is any component of $[-1,1] \smallsetminus S_{n+1}$ then $f_n$ is either constant on $I$, or $f_n$ is linear on $I$ and $f_n(I) \cap S_n = \emptyset$.

A continuum $X$ is \emph{arc-like} (equivalently, \emph{chainable}) if it is homeomorphic to an inverse limit of arcs; that is, if there exist maps $f_n \colon [-1,1] \to [-1,1]$, for $n = 1,2,\ldots$, such that $X \approx \varprojlim \left \langle [-1,1],f_n \right \rangle$.  It is well-known (see e.g.\ \cite{brown1960}) that any arc-like continuum is homeomorphic to the inverse limit of an inverse system $\langle [-1,1],f_n \rangle$ where each bonding map $f_n$ is piecewise-linear.

Let $f_n \colon [-1,1] \to [-1,1]$, $n = 1,2,\ldots$, be a sequence of (piecewise-linear) maps and let $x = \langle x_1,x_2,\ldots \rangle \in X = \varprojlim \left \langle [-1,1],f_n \right \rangle$.  It is well-known that if $x_n = \pm 1$ for infinitely many $n$, then $x$ is an endpoint of $X$ and one can easily embed $X$ in $\mathbb{R}^2$ so as to make $x$ an accessible point (see e.g.\ \cite[p.295]{problems2018}).  Hence we may as well assume that $x_n \neq \pm 1$ for all but finitely many $n$, and in fact by dropping finitely many coordinates we may assume that $x_n \neq \pm 1$ for all $n$.  Then for each $n$, let $h_n \colon [-1,1] \to [-1,1]$ be a (piecewise-linear) homeomorphism such that $h_n(x_n) = 0$, and define $f_n' = h_n \circ f_n \circ h_{n+1}^{-1}$.  Then $\varprojlim \left \langle [-1,1],f_n' \right \rangle \approx X$, and in fact an explicit homeomorphism $h \colon X \to \varprojlim \left \langle [-1,1],f_n' \right \rangle$ is given by $h \left( \langle y_n \rangle_{n=1}^\infty \right) = \langle h_n(y_n) \rangle_{n=1}^\infty$, and clearly $h(x) = \langle 0,0,\ldots \rangle$.  Thus the question of Nadler and Quinn reduces to:

\begin{question}
\label{ques:0 accessible}
If $f_n \colon [-1,1] \to [-1,1]$, $n = 1,2,\ldots$, are piecewise-linear maps such that $f_n(0) = 0$ for each $n$, does there exist an embedding of $X = \varprojlim \left \langle [-1,1],f_n \right \rangle$ into $\mathbb{R}^2$ for which the point $\langle 0,0,\ldots \rangle \in X$ is accessible?
\end{question}

It is this form of the question which we attack in this paper.  Our main result is Theorem~\ref{thm:no twins}, which provides an affirmative answer to Question~\ref{ques:0 accessible} under some additional technical assumptions on the maps $f_n$.

We remark that if, in the above reduction, $\langle [-1,1],f_n \rangle$ is a simplicial system, then so is the system $\langle [-1,1],f_n' \rangle$ for appropriate choices of the homeomorphisms $h_n$.  So an affirmative answer to Question~\ref{ques:0 accessible} for simplicial systems would also imply an affirmative answer to the Nadler-Quinn problem for arc-like continua which are inverse limits of simplicial systems.

\section{Contour factorization}
\label{sec:contour}

In preparation for our study in the following sections of maps $f \colon [-1,1] \to [-1,1]$ with $f(0) = 0$, we begin in this section with the general notions of (one-sided) departures, contour points, and the contour factor of a map $f \colon [0,1] \to [-1,1]$.  In later sections, we will apply these notions to both ``halves'', $f {\restriction}_{[0,1]}$ and $f {\restriction}_{[-1,0]}$, of a map $f \colon [-1,1] \to [-1,1]$.  For the most part, our attention will be on piecewise linear maps, though we will only include this condition when needed.

\begin{defn}
\label{defn:dep}
Let $f \colon [0,1] \to [-1,1]$ be a map with $f(0) = 0$.  A \emph{departure} of $f$ is a number $x > 0$ such that $f(x) \notin f([0,x))$.  We say a departure $x$ of $f$ is \emph{positively oriented} (or a \emph{positive departure}) if $f(x) > 0$, and $x$ is \emph{negatively oriented} (or a \emph{negative departure}) if $f(x) < 0$.

A \emph{contour point} of $f$ is a departure $\alpha$ such that for any departure $x$ of $f$ with $x > \alpha$, there exists a departure $y$ of $f$, of orientation opposite to that of $\alpha$, such that $\alpha < y \leq x$.
\end{defn}

If $f$ is not constant, then it has departures, and therefore has at least one contour point.  Throughout the remainder of this paper, we will always assume our maps are not constant.

A map $f$ may have countably infinitely many contour points, in which case they may be enumerated in decreasing order: $\beta_1 > \beta_2 > \cdots$.  In this case, $\lim_{n \to \infty} f(\beta_n) = 0$, and if $x_0 = \lim_{n \to \infty} \beta_n$ then $f$ is constantly equal to $0$ on $[0,x_0]$.

If $f$ is piecewise-linear, then it has only finitely many contour points, which we may enumerate: $0 < \alpha_1 < \alpha_2 < \ldots < \alpha_n \leq 1$.  For convenience, we also denote $\alpha_0 = 0$.

\begin{defn}
\label{defn:contour factor}
Let $f \colon [0,1] \to [-1,1]$ be a piecewise-linear map with $f(0) = 0$.  The \emph{contour factor} of $f$ is the piecewise-linear map $t_f \colon [0,1] \to [-1,1]$ defined as follows:

Let $0 = \alpha_0 < \alpha_1 < \alpha_2 < \cdots < \alpha_n$ be the contour points of $f$.  Then $t_f \left( \frac{i}{n} \right) = f(\alpha_i)$ for each $i = 0,\ldots,n$, and $t_f$ is defined to be linear in between these points.

A \emph{meandering factor} of $f$ is any (piecewise-linear) map $s \colon [0,1] \to [0,1]$ such that $s(0) = 0$ and $f = t_f \circ s$.
\end{defn}

The next Proposition shows that there always exists at least one meandering factor of any given map $f$.  Note that $f$ may have more than one meandering factor.  Observe that if $t_f$ is the contour factor of $f$ and $s$ is a meandering factor for $f$, then $t_f([0,1]) = f([0,1])$, and $s$ is onto.

\begin{prop}
\label{prop:meandering factor}
Let $f \colon [0,1] \to [-1,1]$ be a piecewise-linear map with $f(0) = 0$, and let $t_f$ be the contour factor of $f$.  Then there exists a meandering factor of $f$.
\end{prop}

\begin{proof}
Let $0 = \alpha_0 < \alpha_1 < \alpha_2 < \cdots < \alpha_n$ be the contour points of $f$.  Given $x \in [0,1]$, define $s(x)$ as follows.

If $x \in [\alpha_{i-1},\alpha_i)$ for some $i \geq 1$, then $f(x)$ is between $f(\alpha_{i-1})$ (inclusive) and $f(\alpha_i)$ (exclusive).  Let $w = w(x) = \frac{f(x) - f(\alpha_{i-1})}{f(\alpha_i) - f(\alpha_{i-1})} \in [0,1)$, so that $f(x) = (1-w) \cdot f(\alpha_{i-1}) + w \cdot f(\alpha_i)$.  Define $s(x) = (1-w) \cdot \frac{i-1}{n} + w \cdot \frac{i}{n}$.  Since $t_f$ is linear on $\left[ \frac{i-1}{n},\frac{i}{n} \right]$,
\begin{align*}
t_f \circ s(x) &= (1-w) \cdot t_f \left( \tfrac{i-1}{n} \right) + w \cdot t_f \left( \tfrac{i}{n} \right) \\
&= (1-w) \cdot f(\alpha_{i-1}) + w \cdot f(\alpha_i) \\
&= f(x) .
\end{align*}

If $x \geq \alpha_n$, then $f(x)$ is between $f(\alpha_{n-1})$ and $f(\alpha_n)$ (inclusive).  Let $w = \frac{f(x) - f(\alpha_{n-1})}{f(\alpha_n) - f(\alpha_{n-1})} \in [0,1]$, so that $f(x) = (1-w) \cdot f(\alpha_{n-1}) + w \cdot f(\alpha_n)$.  Define $s(x) = (1-w) \cdot \frac{n-1}{n} + w \cdot 1$.  Since $t_f$ is linear on $\left[ \frac{n-1}{n},1 \right]$,
\begin{align*}
t_f \circ s(x) &= (1-w) \cdot t_f \left( \tfrac{n-1}{n} \right) + w \cdot t_f(1) \\
&= (1-w) \cdot f(\alpha_{n-1}) + w \cdot f(\alpha_n) \\
&= f(x) .
\end{align*}

It is easy to see that $s$ is continuous and piecewise-linear (since $f$ is), and $s(0) = 0$.
\end{proof}

\begin{lem}
\label{lem:same contour}
Let $f,g \colon [0,1] \to [-1,1]$ be piecewise-linear maps with $f(0) = g(0) = 0$.  Then $t_f = t_g$ if and only if:
\begin{enumerate}
\item for each departure $x$ of $f$, there exists a departure $x'$ of $g$ such that $f([0,x)) = g([0,x'))$; and
\item for each departure $x'$ of $g$, there exists a departure $x$ of $f$ such that $f([0,x)) = g([0,x'))$.
\end{enumerate}
\end{lem}

\begin{proof}
Let $0 = \alpha_0 < \alpha_1 < \alpha_2 < \cdots < \alpha_n$ be the contour points of $f$.

Suppose $t_f = t_g$.  Then $f$ and $g$ have the same number of contour points, which have the same images.  Let $0 = \beta_0 < \beta_1 < \beta_2 < \cdots < \beta_n$ be the contour points of $g$, so that $f(\alpha_i) = g(\beta_i)$ for each $i$.  Let $x$ be a departure of $f$.  Note that $0 < x \leq \alpha_n$, so there exists $i \geq 1$ such that $x \in (\alpha_{i-1},\alpha_i]$.  Note that since $\alpha_{i-1}$ is a contour point, there exists a departure $y \in (\alpha_{i-1},x]$ with orientation opposite that of $\alpha_{i-1}$.  If $x$ were to have the same orientation as $\alpha_{i-1}$, then it would follow that there is another contour point $\alpha \in [y,x)$.  However, there are no contour points strictly between $\alpha_{i-1}$ and $\alpha_i$.  Therefore, $x$ must have the same orientation as $\alpha_i$.  Also, $f([0,x)) = f([\alpha_{i-1},x))$ is the interval between $f(\alpha_{i-1})$ (inclusive) and $f(x)$ (exclusive).  Similarly, for any departure $x' \in (\beta_{i-1},\beta_i]$ of $g$, $x'$ has the same orientation as $\beta_i$, and $g([0,x')) = g([\beta_{i-1},x'))$ is the interval between $g(\beta_{i-1}) = f(\alpha_{i-1})$ (inclusive) and $g(x')$ (exclusive).  In the same way, considering the departures $\alpha_{i-1}$ of $f$ and $\beta_{i-1}$ of $g$, we deduce that $g([0,\beta_{i-1}]) = f([0,\alpha_{i-1}])$.

Since $f(x)$ is between $f(\alpha_{i-1}) = g(\beta_{i-1})$ (exclusive) and $f(\alpha_i) = g(\beta_i)$ (inclusive), there exists $x' \in (\beta_{i-1},\beta_i]$ such that $g(x') = f(x)$.  Choose $x'$ minimal with respect to this property, so that $g(x') \notin g((\beta_{i-1},x'))$.  Also,
\[ g(x') = f(x) \notin f([0,x)) \supset f([0,\alpha_{i-1}]) = g([0,\beta_{i-1}]) .\]
Therefore $x'$ is a departure of $g$, and $f([0,x)) = g([0,x'))$ as these are both the interval between $f(\alpha_{i-1}) = g(\beta_{i-1})$ (inclusive) and $f(x) = g(x')$ (exclusive).  This proves (1); (2) follows similarly.

Conversely, suppose (1) and (2) hold.  Fix any $i \geq 1$.  By assumption, there exists a departure $\beta_i$ of $g$ such that $g([0,\beta_i)) = f([0,\alpha_i))$, which in particular implies that $\beta_i$ has the same orientation for $g$ as $\alpha_i$ has for $f$.  Let $x' > \beta_i$ be any departure of $g$.  By (2), there exists a departure $x$ of $f$ such that $f([0,x)) = g([0,x'))$, and clearly $x > \alpha_i$.  Since $\alpha_i$ is a contour point of $f$, there exists a departure $y \in (\alpha_i,x]$ of $f$ of orientation opposite to that of $\alpha_i$.  Then by (1), there exists a departure $y'$ of $g$ such that $f([0,y)) = g([0,y'))$, which implies $y' \in (\beta_i,x']$ and $y'$ has opposite orientation to that of $\beta_i$.  Therefore $\beta_i$ is a contour point of $g$.  Thus for each contour point $\alpha_i$ of $f$ there exists a corresponding contour point $\beta_i$ of $g$ with $f(\alpha_i) = g(\beta_i)$, and clearly $\beta_{i-1} < \beta_i$ for each $i$.  Likewise, for each contour point of $g$ there is a corresponding contour point of $f$ with the same image, and these appear in the same order in $f$ as they do in $g$.  Thus $t_f = t_g$.
\end{proof}

We give two applications of Lemma~\ref{lem:same contour} below.  The first, Proposition~\ref{prop:contour minimal}, shows that the contour factor is the simplest map $\tau \colon [0,1] \to [0,1]$ for which there exists $\sigma \colon [0,1] \to [0,1]$ with $\sigma(0) = 0$ and $f = \tau \circ \sigma$.  This is a characteristic property of the contour factor, however we will not use this result in the remainder of this paper.  The second, Lemma~\ref{lem:same comp contour}, shows that the contour factor of a composition $g \circ f$ depends only on $g$ and the contour factor of $f$.

\begin{prop}
\label{prop:contour minimal}
Let $f \colon [0,1] \to [-1,1]$ be a piecewise-linear map with $f(0) = 0$.  Suppose $\tau \colon [0,1] \to [-1,1]$ and $\sigma \colon [0,1] \to [0,1]$ are piecewise-linear maps such that $\tau(0) = \sigma(0) = 0$, $\sigma$ is onto, and $f = \tau \circ \sigma$.  Then $t_\tau = t_f$.  Consequently, there exists $s_0 \colon [0,1] \to [0,1]$ such that $f = t_f \circ s_0 \circ \sigma$ (that is, $s_0 \circ \sigma$ is a meandering factor of $f$).
\end{prop}

\begin{proof}
To prove $t_\tau = t_f$, we use Lemma~\ref{lem:same contour}.

Let $x$ be any departure of $f$, and let $x' = \sigma(x)$, so that $\tau(x') = f(x)$.  Then $x' \notin \sigma([0,x))$ since otherwise $f(x) = \tau(x')$ would be in $f([0,x)) = \tau(\sigma([0,x))$, contradicting the assumption that $x$ is a departure of $f$.  This means $\sigma([0,x)) = [0,x')$.  Now $\tau([0,x')) = \tau(\sigma([0,x))) = f([0,x))$, and $x'$ is a departure of $\tau$ since $\tau(x') = f(x) \notin f([0,x)) = \tau([0,x'))$.

Conversely, let $x'$ be any departure of $\tau$.  Since $\sigma$ is onto, there exists $x$ such that $\sigma(x) = x'$.  Choose $x$ minimal with respect to this property, so that $\sigma([0,x)) = [0,\sigma(x))$.  Then $f([0,x)) = \tau(\sigma([0,x))) = \tau([0,\sigma(x))) = \tau([0,x'))$, and $x$ is a departure of $f$ since $f(x) = \tau(x') \notin \tau([0,x')) = f([0,x))$.

Therefore, by Lemma~\ref{lem:same contour}, $t_\tau = t_f$.
\end{proof}

\begin{lem}
\label{lem:same comp contour}
Let $f_1,f_2 \colon [0,1] \to [-1,1]$ and $g \colon [-1,1] \to [-1,1]$ be piecewise-linear maps with $f_1(0) = f_2(0) = g(0) = 0$.  If $t_{f_1} = t_{f_2}$, then $t_{g \circ f_1} = t_{g \circ f_2}$.
\end{lem}

\begin{proof}
Suppose $t_{f_1} = t_{f_2}$.  To prove $t_{g \circ f_1} = t_{g \circ f_2}$, we use Lemma~\ref{lem:same contour}.

Let $x$ be a departure of $g \circ f_1$.  Then $x$ is also a departure of $f_1$, therefore by Lemma~\ref{lem:same contour}, there exists a departure $x'$ of $f_2$ such that $f_1([0,x)) = f_2([0,x'))$, and so $g \circ f_1([0,x)) = g \circ f_2([0,x'))$.  Suppose $y' \in [0,x')$ is such that $g \circ f_2(y') = g \circ f_2(x')$.  Then $f_2(y') \in f_2([0,x')) = f_1([0,x))$, so there exists $y \in [0,x)$ such that $f_1(y) = f_2(y')$.  But then $g \circ f_1(y) = g \circ f_2(y') = g \circ f_2(x') = g \circ f_1(x)$, contradicting the assumption that $x$ is a departure of $g \circ f_1$.  Thus $x'$ is a departure of $g \circ f_2$.

Similarly, for each departure $x'$ of $g \circ f_2$ there exists a departure $x$ of $g \circ f_1$ such that $g \circ f_1([0,x)) = g \circ f_2([0,x'))$.  Therefore, by Lemma~\ref{lem:same contour}, $t_{g \circ f_1} = t_{g \circ f_2}$.
\end{proof}

\section{Radial departures}
\label{sec:rad deps}

In the remainder of this paper we consider maps $f \colon [-1,1] \to [-1,1]$ satisfying $f(0) = 0$, and for which $f {\restriction}_{[-1,0]}$ and $f {\restriction}_{[0,1]}$ are both non-constant (we will tacitly assume this latter condition throughout the remainder of this paper).  We adapt the above concepts of departures and contour points to both the left and right ``halves'' of such a function $f$, as follows.

Define $\mathsf{r} \colon [0,1] \to [-1,0]$ by $\mathsf{r}(x) = -x$.

\begin{defn}
Let $f \colon [-1,1] \to [-1,1]$ be a map with $f(0) = 0$.
\begin{itemize}
\item A \emph{right departure} of $f$ is a number $x > 0$ such that $x$ is a departure (in the sense of Definition~\ref{defn:dep}) of $f {\restriction}_{[0,1]}$.
\item A \emph{left departure} of $f$ is a number $x < 0$ such that $-x$ is a departure (in the sense of Definition~\ref{defn:dep}) of $f {\restriction}_{[-1,0]} \circ \mathsf{r}$; i.e.\ a number $x < 0$ such that $f(x) \notin f((x,0])$.
\end{itemize}
If $x$ is either a right departure or a left departure of $f$, then we say $x$ is \emph{positively oriented} if $f(x) > 0$, and $x$ is \emph{negatively oriented} if $f(x) < 0$.
\begin{itemize}
\item A \emph{right contour point} of $f$ is a number $\alpha > 0$ such that $\alpha$ is a contour point (in the sense of Definition~\ref{defn:dep}) of $f {\restriction}_{[0,1]}$.
\item A \emph{left contour point} of $f$ is a number $\beta < 0$ such that $-\beta$ is a contour point (in the sense of Definition~\ref{defn:dep}) of $f {\restriction}_{[-1,0]} \circ \mathsf{r}$.
\end{itemize}
\end{defn}

\begin{defn}
Let $f \colon [-1,1] \to [-1,1]$ be a map with $f(0) = 0$.  A \emph{radial departure} of $f$ is a pair $\langle x_1,x_2 \rangle$ such that $-1 \leq x_1 < 0 < x_2 \leq 1$ and either:
\begin{enumerate}[label=(\arabic{*})]
\item \label{pos dep} $f((x_1,x_2)) = (f(x_1),f(x_2))$; or
\item \label{neg dep} $f((x_1,x_2)) = (f(x_2),f(x_1))$.
\end{enumerate}
We say a radial departure $\langle x_1,x_2 \rangle$ of $f$ is \emph{positively oriented} (or a \emph{positive radial departure}) if \ref{pos dep} holds, and $\langle x_1,x_2 \rangle$ is \emph{negatively oriented} (or a \emph{negative radial departure}) if \ref{neg dep} holds.
\end{defn}

Observe that if $\langle x_1,x_2 \rangle$ is a positive (respectively, negative) radial departure of $f$, then $x_1$ is a negative (respectively, positive) left departure of $f$ and $x_2$ is a positive (respectively, negative) right departure of $f$.  However, the converse is not true: if $x_1$ is a left departure of $f$ and $x_2$ is a right departure of $f$, and these two departures have opposite orientations, it is not necessarily the case that $\langle x_1,x_2 \rangle$ is a radial departure of $f$, because it may happen that $f(x_2) \in f((x_1,0])$ or $f(x_1) \in f([0,x_2))$.

We develop several basic properties of radial departures in the remainder of this section.

\begin{lem}
\label{lem:dep values unique}
Let $f \colon [-1,1] \to [-1,1]$ be a map with $f(0) = 0$.  Suppose $\langle x_1,x_2 \rangle$ and $\langle x_1',x_2' \rangle$ are radial departures of $f$.
\begin{enumerate}
\item If $\{f(x_1),f(x_2)\} = \{f(x_1'),f(x_2')\}$ then $x_1 = x_1'$ and $x_2 = x_2'$.
\item If $\langle x_1,x_2 \rangle$ and $\langle x_1',x_2' \rangle$ have opposite orientations, then $\{f(x_1),f(x_2)\} \cap \{f(x_1'),f(x_2')\} = \emptyset$.
\end{enumerate}
\end{lem}

\begin{proof}
For (1), suppose $\{f(x_1),f(x_2)\} = \{f(x_1'),f(x_2')\}$.  Since $\langle x_1,x_2 \rangle$ is a radial departure, it follows that $x_1',x_2' \notin (x_1,x_2)$, so $x_1' \leq x_1$ and $x_2' \geq x_2$.  Similarly, $x_1 \leq x_1'$ and $x_2 \geq x_2'$.

For (2), suppose $\langle x_1,x_2 \rangle$ is a positive radial departure and $\langle x_1',x_2' \rangle$ is a negative radial departure.  This means that $f(x_1) < 0 < f(x_1')$ and $f(x_2') < 0 < f(x_2)$.  Suppose for a contradiction that $f(x_1') = f(x_2)$.  Then $x_1' < x_1$ since $f(x_2) \notin f((x_1,x_2))$ and $x_2' < x_2$ since $f(x_1') \notin f((x_1',x_2'))$.  But now $x_1 \in (x_1',x_2')$ and $x_2' \in (x_1,x_2)$, so $f(x_1) \in f((x_1',x_2')) = (f(x_2'),f(x_1'))$ and $f(x_2') \in f((x_1,x_2)) = (f(x_1),f(x_2))$, implying $f(x_1) > f(x_2')$ and $f(x_2') > f(x_1)$, a contradiction.  Therefore $f(x_1') \neq f(x_2)$.  Similarly, $f(x_2') \neq f(x_1)$.
\end{proof}

\begin{prop}
\label{prop:meet join}
Let $f \colon [-1,1] \to [-1,1]$ be a map with $f(0) = 0$.  Suppose $\langle x_1,x_2 \rangle$ and $\langle x_1',x_2' \rangle$ are radial departures of $f$ with the same orientation.  Then both
\[ \langle \min\{x_1,x_1'\}, \max\{x_2,x_2'\} \rangle \quad \textrm{and} \quad \langle \max\{x_1,x_1'\}, \min\{x_2,x_2'\} \rangle \]
are radial departures of $f$ of the same orientation as $\langle x_1,x_2 \rangle$ (and $\langle x_1',x_2' \rangle$).
\end{prop}

\begin{proof}
Suppose $\langle x_1,x_2 \rangle$ and $\langle x_1',x_2' \rangle$ are both positive radial departures (the case where they are both negative is similar).  Let $y_1 = \min\{x_1,x_1'\}$, $y_2 = \max\{x_2,x_2'\}$, and let $z_1 = \max\{x_1,x_1'\}$, $z_2 = \min\{x_2,x_2'\}$.  Observe that $f(y_1) = \min\{f(x_1),f(x_1')\}$, $f(y_2) = \max\{f(x_2),f(x_2')\}$, $f(z_1) = \max\{f(x_1),f(x_1')\}$, and $f(z_2) = \min\{f(x_2),f(x_2')\}$.

Given $x \in (y_1,y_2)$, we have $x > x_1$ or $x > x_1'$, hence $f(x) > f(x_1)$ or $f(x) > f(x_1')$, so $f(x) > f(y_1)$.  Likewise, $f(x) < f(y_2)$.  Thus $f((y_1,y_2)) = (f(y_1),f(y_2))$.

Similarly, given $x \in (z_1,z_2)$, we have $x > x_1$ and $x > x_1'$, hence $f(x) > f(x_1)$ and $f(x) > f(x_1')$, so $f(x) > f(z_1)$.  Likewise, $f(x) < f(z_2)$.  Thus $f((z_1,z_2)) = (f(z_1),f(z_2))$.
\end{proof}

\begin{prop}
\label{prop:alt dep nested}
Let $f \colon [-1,1] \to [-1,1]$ be a map with $f(0) = 0$.  Suppose $\langle x_1,x_2 \rangle$ and $\langle x_1',x_2' \rangle$ are radial departures of $f$ with opposite orientations.  Then either
\[ x_1 < x_1' < 0 < x_2' < x_2 \quad \textrm{or} \quad x_1' < x_1 < 0 < x_2 < x_2' .\]
\end{prop}

\begin{proof}
Suppose $\langle x_1,x_2 \rangle$ is a positive radial departure and $\langle x_1',x_2' \rangle$ is a negative radial departure.  If $x_1 < x_1'$, then $x_1' \in (x_1,x_2)$, so $f(x_1') \in (f(x_1),f(x_2))$, so that $f(x_1') < f(x_2)$.  Now we must have $x_2' < x_2$, since otherwise we would have $x_2 \in (x_1',x_2')$ and then $f(x_2) \in f((x_1',x_2')) = (f(x_2'),f(x_1'))$, meaning $f(x_2) < f(x_1')$, a contradiction.  Similarly, if $x_2' < x_2$ then $x_1 < x_1'$.
\end{proof}

\begin{prop}
\label{prop:comp dep}
Let $f,g \colon [-1,1] \to [-1,1]$ be maps with $f(0) = g(0) = 0$, and let $x_1,x_2$ be such that $-1 \leq x_1 < 0 < x_2 \leq 1$.  Then $\langle x_1,x_2 \rangle$ is a positive (respectively, negative) radial departure of $f \circ g$ if and only if either:
\begin{enumerate}
\item $\langle x_1,x_2 \rangle$ is a positive radial departure of $g$ and $\langle g(x_1),g(x_2) \rangle$ is a positive (respectively, negative) radial departure of $f$; or
\item $\langle x_1,x_2 \rangle$ is a negative radial departure of $g$ and $\langle g(x_2),g(x_1) \rangle$ is a negative (respectively, positive) radial departure of $f$.
\end{enumerate}
\end{prop}

\begin{proof}
If $\langle x_1,x_2 \rangle$ is a positive radial departure of $g$ and $\langle g(x_1),g(x_2) \rangle$ is a positive (respectively, negative) radial departure of $f$, then $f \circ g((x_1,x_2)) = f((g(x_1),g(x_2)))$, which equals $(f \circ g(x_1),f \circ g(x_2))$ (respectively, $(f \circ g(x_2),f \circ g(x_1))$).  Thus $\langle x_1,x_2 \rangle$ is a positive (respectively, negative) radial departure of $f \circ g$.

Similarly, if $\langle x_1,x_2 \rangle$ is a negative radial departure of $g$ and $\langle g(x_2),g(x_1) \rangle$ is a negative (respectively, positive) radial departure of $f$, then $f \circ g((x_1,x_2)) = f((g(x_2),g(x_1)))$, which equals $(f \circ g(x_1),f \circ g(x_2))$ (respectively, $(f \circ g(x_2),f \circ g(x_1))$).  Thus $\langle x_1,x_2 \rangle$ is a positive (respectively, negative) radial departure of $f \circ g$.

Conversely, suppose $\langle x_1,x_2 \rangle$ is a radial departure of $f \circ g$.  Then $g(x_1),g(x_2) \notin g((x_1,x_2))$ since $f \circ g(x_1),f \circ g(x_2) \notin f \circ g((x_1,x_2))$, and $g(x_1) \neq g(x_2)$ since $f \circ g(x_1) \neq f \circ g(x_2)$.  It follows that $\langle x_1,x_2 \rangle$ is a radial departure of $g$.  Now since $f \circ g(x_1),f \circ g(x_2) \notin f \circ g((x_1,x_2))$ and $f \circ g(x_1) \neq f \circ g(x_2)$, we deduce:
\begin{enumerate}
\item if $\langle x_1,x_2 \rangle$ is a positive radial departure of $g$ then $\langle g(x_1),g(x_2) \rangle$ is a radial departure of $f$, and by the above argument the orientation of the radial departure $\langle g(x_1),g(x_2) \rangle$ of $f$ matches the orientation of the radial departure $\langle x_1,x_2 \rangle$ of $f$; and
\item if $\langle x_1,x_2 \rangle$ is a negative radial departure of $g$ then $\langle g(x_2),g(x_1) \rangle$ is a radial departure of $f$, and by the above argument the orientation of the radial departure $\langle g(x_2),g(x_1) \rangle$ of $f$ is opposite to the orientation of the radial departure $\langle x_1,x_2 \rangle$ of $f$.
\end{enumerate}
\end{proof}

\section{Radial departures and embeddings}
\label{sec:embeddings}

The main result of this section, Proposition~\ref{prop:embed 0 accessible}, essentially appears in \cite[Lemma 7.2]{anusic-bruin-cinc2020}, using a different language.  We present an alternative argument here, to keep this treatment self-contained.

Let $\mathbb{H} = \{\langle x,y \rangle \in \mathbb{R}^2: x \geq 0\}$ denote the closed right half plane.  We use the usual Euclidean metric on the (half) plane, $d \left( \langle x_1,y_1 \rangle, \langle x_2,y_2 \rangle \right) = \|\langle x_1,y_1 \rangle - \langle x_2,y_2 \rangle\|$, where $\|p\|$ denotes the Euclidean norm of the point $p \in \mathbb{R}^2$.

\begin{lem}
\label{lem:tuck}
Suppose $f \colon [-1,1] \to [-1,1]$ is a map with $f(0) = 0$ all of whose radial departures have the same orientation.  Then for any $\varepsilon > 0$ there exists an embedding $\Phi$ of $[-1,1]$ into $\mathbb{H}$ such that $\Phi(0) = \langle 0,0 \rangle$, $\Phi([-1,1]) \cap \partial \mathbb{H} = \{\langle 0,0 \rangle\}$, and $\|\Phi(x) - \langle 0,f(x) \rangle\| < \varepsilon$ for all $x \in [-1,1]$.
\end{lem}

\begin{proof}
Suppose $f$ has no negative radial departures (the case where $f$ has no positive radial departures is similar).

Without loss of generality, we may assume that $f$ has only finitely many contour points.  Indeed, if $f$ has infinitely many contour points, we may approximate $f$ arbitrarily closely by maps with only finitely many contour points, all of whose radial departures are positive, e.g.\ by making $f$ monotone between some negative left contour point and some positive right contour point.

The proof consists of three steps.  In the first step, we show that we can simultaneously traverse the graphs of $f {\restriction}_{[0,1]}$ (from $0$ to $1$) and $f {\restriction}_{[-1,0]}$ (from $0$ to $-1$) in such a way that the point on the right side is always above (or equal to) the point on the left side.  In the second step, we perturb $f$ a small amount so that during this traversal the point on the right side is always strictly above the point on the left side.  This enables us to embed an arc in $\mathbb{H}$ with both the right and left sides of the graph of this perturbed $f$ laid out to the right starting from $\langle 0,0 \rangle$; however, this arc may have some constant (horizontal) sections at some contour points of $f$.  In the third step, we contract these horizontal segments to points, to obtain the desired embedding $\Phi$.  Figure~\ref{fig:tuck} illustrates the outcomes of these steps for a sample function $f$.

\begin{figure}
\begin{center}
\includegraphics{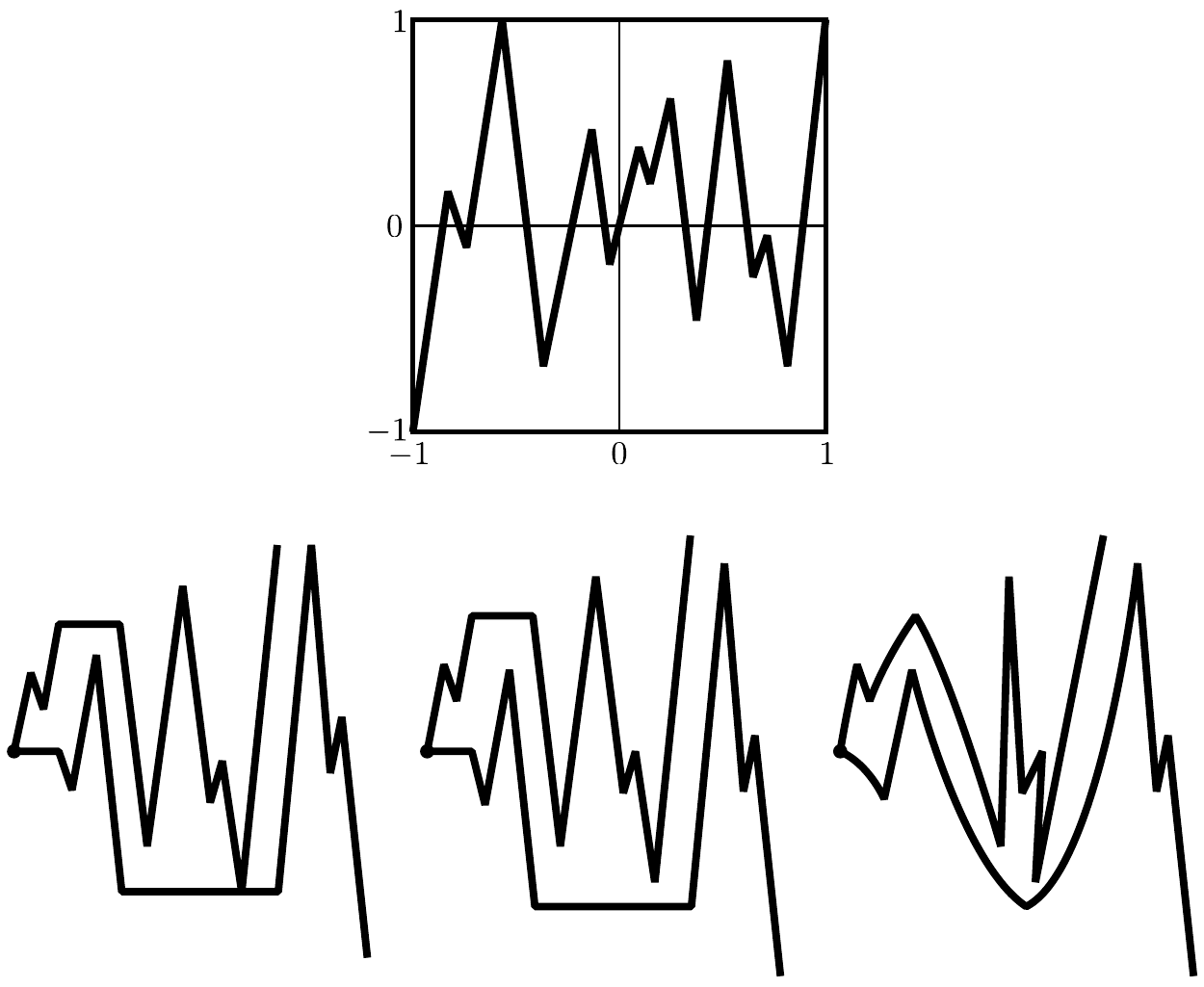}
\end{center}

\caption{Above: a sample function $f$ with only positive radial departures.  Below: the outcomes of the three steps of the proof of Lemma~\ref{lem:tuck}.}
\label{fig:tuck}
\end{figure}

\medskip
\noindent \textbf{Step 1.}  We show that it is possible to traverse the right and left sides of the graph of $f$ simultaneously in such a way that the $y$-value of the point on the right is always greater or equal to the $y$-value of the point on the left.  Precisely, we show below that there exist onto, monotone maps $\psi_+ \colon [0,r_+] \to [0,1]$ and $\psi_- \colon [0,r_-] \to [-1,0]$, for some $r_+,r_- > 0$, such that $\psi_+(0) = \psi_-(0) = 0$, $f(\psi_+(t)) \geq f(\psi_-(t))$ for all $t \in [0, \min\{r_+,r_-\}]$, and these maps are not locally constant except possibly at points $t$ for which the value is a contour point of $f$.  The bulk of the work involved is to show that it is possible to completely traverse one side of the graph of $f$, while simultaneously traversing a part of the other side in this way, which is the content of the following Claim.

\begin{claim}
\label{claim:psi+ psi-}
There exist maps $\psi_+ \colon [0,r] \to [0,1]$ and $\psi_- \colon [0,r] \to [-1,0]$, for some $r > 0$, such that:
\begin{enumerate}
\item $\psi_+(0) = \psi_-(0) = 0$, $\psi_+$ is non-decreasing, and $\psi_-$ is non-increasing;
\item For each $x \in (0,1]$ (respectively, $x \in [-1,0)$), if $\psi_+^{-1}(x)$ (respectively, $\psi_-^{-1}(x)$) is not a singleton, then $x$ is a positive right (respectively, negative left) contour point of $f$ (possibly $x = 0$);
\item $f(\psi_+(t)) \geq f(\psi_-(t))$ for all $t$;
\item Either $\psi_+$ or $\psi_-$ is onto.
\end{enumerate}
\end{claim}

\begin{proof}[Proof of Claim~\ref{claim:psi+ psi-}]
\renewcommand{\qedsymbol}{\textsquare (Claim~\ref{claim:psi+ psi-})}
We proceed by induction on the number of right contour points of $f$.  To assist with the induction, we add one further property to the above four:
\begin{enumerate}
\setcounter{enumi}{4}
\item If $\psi_-$ is not onto, then $\psi_-(r)$ is either $0$ or a negative left contour point of $f$.
\end{enumerate}

For the base case, suppose $f$ has only one (non-zero) right contour point $\alpha$.  If $f(\alpha) > 0$, then it suffices to let $\psi_+$ linearly parameterize $[0,1]$ and let $\psi_- \equiv 0$.  Suppose now that $f(\alpha) < 0$.  Then first (non-zero) left contour point $\beta$ of $f$ must be negative, since otherwise $\langle \beta,\alpha \rangle$ would be a negative radial departure of $f$.  If $\beta$ is the only left contour point of $f$, then it suffices to let $\psi_+ \equiv 0$ and let $\psi_-$ linearly parameterize $[-1,0]$.  Otherwise, $f(\beta) \leq f(\alpha)$, since otherwise $\langle \beta',\alpha \rangle$ would be a negative radial departure of $f$, where $\beta'$ is the next left contour point of $f$ after $\beta$ (going from $0$ towards $-1$).  In this case, let $0 < r_1 < r_2$, let $\psi_+{\restriction}_{[0,r_1]} \equiv 0$ and let $\psi_-{\restriction}_{[0,r_1]}$ linearly parameterize $[\beta,0]$, then let $\psi_+{\restriction}_{[r_1,r_2]}$ linearly parameterize $[0,1]$ while $\psi_-{\restriction}_{[r_1,r_2]} \equiv \beta$.

For the inductive step, suppose the Claim is established for maps with $n$ right contour points, and suppose $f$ has $n+1$ right contour points.  Let $\alpha_n$,$\alpha_{n+1}$ be the last two right contour points of $f$, and apply induction to $f{\restriction}_{[0,\alpha_n]}$ to get $\psi_+ \colon [0,r] \to [0,\alpha_n]$ and $\psi_- \colon [0,r] \to [-1,0]$ satisfying the properties (1)--(5).  If $\psi_-$ is onto, then we are already done.  Suppose then that $\psi_-$ is not onto, which means that $\psi_+(r) = \alpha_n$ and $\psi_-(r)$ is a negative left contour point $\beta$ of $f$.  Let $r < r_1 < r_2$.  If $f(\alpha_{n+1}) \geq f(\beta)$, then it suffices to let $\psi_+ {\restriction}_{[r,r_1]}$ linearly parameterize $[\alpha_n,1]$ and let $\psi_- {\restriction}_{[r,r_1]} \equiv \beta$.  Suppose now that $f(\alpha_{n+1}) < f(\beta)$, which means $\alpha_n$ is a positive right contour point.  If $f(x) \leq f(\alpha_n)$ for all $x \in [-1,\beta]$, then it suffices to let $\psi_+ {\restriction}_{[r,r_1]} \equiv \alpha_n$ and let $\psi_- {\restriction}_{[r,r_1]}$ linearly parameterize $[-1,\beta]$.  On the other hand, if $f(x) > f(\alpha_n)$ for some $x \in [-1,\beta)$, then there must be a left contour point $\beta' \in (x,\beta)$ with $f(\beta') \leq f(\alpha_{n+1})$ and $f(x') \leq f(\alpha_n)$ for all $x' \in [\beta',\beta]$, since otherwise $\langle x,\alpha_{n+1} \rangle$ would be a negative radial departure of $f$.  In this case, let $\psi_+{\restriction}_{[r,r_1]} \equiv \alpha_n$ and let $\psi_-{\restriction}_{[r,r_1]}$ linearly parameterize $[\beta',\beta]$, then let $\psi_+{\restriction}_{[r_1,r_2]}$ linearly parameterize $[\alpha_n,1]$ while $\psi_-{\restriction}_{[r_1,r_2]} \equiv \beta'$.

This completes the inductive argument.
\end{proof}

Now we may assume $r < \frac{\varepsilon}{2}$.  If $\psi_+$ is not onto, let $r_- = r$, let $r < r_+ < \frac{\varepsilon}{2}$ and let $\psi_+ {\restriction}_{[r,r_+]}$ linearly parameterize $[\psi_+(r),1]$.  If $\psi_-$ is not onto, let $r_+ = r$, let $r < r_- < \frac{\varepsilon}{2}$ and let $\psi_- {\restriction}_{[r,r_-]}$ linearly parameterize $[-1,\psi_-(r)]$.  If both $\psi_+$ and $\psi_-$ are onto, then let $r_+ = r_- = r$.

\medskip
\noindent \textbf{Step 2.}
By a small perturbation of $f$, we may construct a map $f' \colon [-1,1] \to [-1,1]$ with $f'(0) = 0$ such that $f'(x) > f(x)$ for all $x \in (0,1]$, $f'(x) < f(x)$ for all $x \in [-1,0)$, and $\|\langle \frac{\varepsilon}{2},f'(x) \rangle - \langle 0,f(x) \rangle\| < \varepsilon$ for all $x \in [-1,1]$.  Define $\Psi_+ \colon [0,r_+] \to \mathbb{H}$ and $\Psi_- \colon [0,r_-] \to \mathbb{H}$ by $\Psi_+(t) = \langle t,f'(\psi_+(t)) \rangle$ and $\Psi_-(t) = \langle t,f'(\psi_-(t)) \rangle$.  Note that $\Psi_+(t_1) = \Psi_-(t_2)$ only when $t_1 = t_2 = 0$, hence $\Psi_+([0,r_+]) \cup \Psi_-([0,r_-])$ is an arc in $\mathbb{H}$.

\medskip
\noindent \textbf{Step 3.}
Let $M \colon \mathbb{H} \to \mathbb{H}$ be a monotone map such that $M(\langle 0,0 \rangle) = \langle 0,0 \rangle$, the $y$-coordinate of $M(p)$ equals the $y$-coordinate of $p$ for each $p \in \mathbb{H}$, $M([0,\frac{\varepsilon}{2}] \times \mathbb{R}) \subseteq [0,\frac{\varepsilon}{2}] \times \mathbb{R}$, $M(\Psi_+(\psi_+^{-1}(x)))$ (respectively, $M(\Psi_-(\psi_-^{-1}(x)))$) is a singleton for each $x \in [0,1]$ (respectively, $x \in [-1,0]$), and $M$ is otherwise one-to-one.  Finally, define $\Phi \colon [-1,1] \to M(\Psi_+([0,r_+]) \cup \Psi_-([0,r_-]))$ by $\Phi(x) = M(\Psi_+(\psi_+^{-1}(x)))$ if $x \in [0,1]$ and $\Phi(x) = M(\Psi_-(\psi_-^{-1}(x)))$ if $x \in [-1,0]$.  This $\Phi$ has the desired properties.
\end{proof}

We remark that the converse of Lemma~\ref{lem:tuck} is also true: if $f$ has both a positive and a negative radial departure, then for some $\varepsilon > 0$ there does not exist a map $\Phi$ as in the Lemma.  However, we will not need this result, so we do not include a proof; see also \cite[Lemma 7.6]{anusic-bruin-cinc2020}.

\begin{prop}
\label{prop:embed 0 accessible}
Let $f_n \colon [-1,1] \to [-1,1]$, $n = 1,2,\ldots$, be maps with $f_n(0) = 0$ for each $n$.  Suppose that for each $n$, all radial departures of $f_n$ have the same orientation.  Then there exists an embedding of $X = \varprojlim \left \langle [-1,1], f_n \right \rangle$ into $\mathbb{R}^2$ for which the point $\langle 0,0,\ldots \rangle \in X$ is accessible.
\end{prop}

\begin{proof}
We apply the Anderson-Choquet Embedding Theorem \cite{anderson-choquet1959}.  It suffices to prove that for any embedding $\Omega_n \colon [-1,1] \to \mathbb{H}$ such that $\Omega_n(0) = \langle 0,0 \rangle$, and for any $\varepsilon > 0$, there exists an embedding $\Omega_{n+1} \colon [-1,1] \to \mathbb{H}$ such that $\Omega_{n+1}(0) = \langle 0,0 \rangle$ and $\|\Omega_{n+1}(x) - \Omega_n(f_n(x))\| < \varepsilon$ for all $x \in [-1,1]$.

Let $A$ denote the straight segment in $\mathbb{R}^2$ from $\langle 0,0 \rangle$ to $\langle -1,0 \rangle$, and let $I$ denote the straight segment in $\mathbb{R}^2$ from $\langle 0,-1 \rangle$ to $\langle 0,1 \rangle$.

Suppose $\Omega_n \colon [-1,1] \to \mathbb{H}$ is an embedding such that $\Omega_n(0) = \langle 0,0 \rangle$, and let $\varepsilon > 0$ be arbitrary.  Let $H \colon \mathbb{R}^2 \to \mathbb{R}^2$ be a homeomorphism such that $H(\Omega_n(x)) = \langle 0,x \rangle$ for each $x \in [-1,1]$, and $H {\restriction}_A = \mathrm{id}_A$.  For some small $\varepsilon' > 0$, apply Lemma~\ref{lem:tuck} to construct an embedding $\Phi \colon [-1,1] \to \mathbb{H}$ such that (1) $\Phi(0) = \langle 0,0 \rangle$; (2) $\Phi([-1,1]) \cap \partial \mathbb{H} = \{\langle 0,0 \rangle\}$; and (3) $\|\Phi(x) - \langle 0,f_n(x) \rangle\| < \varepsilon'$ for all $x \in [-1,1]$.  Let $\Omega_{n+1} = H^{-1} \circ \Phi$.  Clearly $\Omega_{n+1}(0) = \langle 0,0 \rangle$.  For sufficiently small $\varepsilon'$, we have (from (2)) that $\Omega_{n+1}$ is an embedding of $[-1,1]$ into $\mathbb{H}$, and (from (3) and uniform continuity of $H$ in a neighborhood of $A \cup I$) $\|\Omega_{n+1}(x) - \Omega_n(f_n(x))\| < \varepsilon$ for all $x \in [-1,1]$, as desired.
\end{proof}

\section{Radial contour factorization}
\label{sec:rad contour}

In this section we apply the contour factorization from Section~\ref{sec:prelim} to both the left and right ``halves'' of a map $f \colon [-1,1] \to [-1,1]$ with $f(0) = 0$.  Recall that we will always assume that $f {\restriction}_{[-1,0]}$ and $f {\restriction}_{[0,1]}$ are both non-constant.  The resulting factorization will be used in Section~\ref{sec:no twins} to produce alternative inverse limit representations of a given arc-like continuum.

Recall that $\mathsf{r} \colon [0,1] \to [-1,0]$ is the function $\mathsf{r}(x) = -x$.

\begin{defn}
\label{defn:rad contour factor}
Let $f \colon [-1,1] \to [-1,1]$ be a piecewise-linear map with $f(0) = 0$.  The \emph{radial contour factor} of $f$ is the piecewise-linear map $t_f \colon [-1,1] \to [-1,1]$ such that:
\begin{enumerate}
\item $t_f {\restriction}_{[0,1]}$ is the contour factor of $f {\restriction}_{[0,1]}$ (in the sense of Definition~\ref{defn:contour factor}), and;
\item $t_f {\restriction}_{[-1,0]} \circ \mathsf{r}$ is the contour factor of $f {\restriction}_{[-1,0]} \circ \mathsf{r}$ (in the sense of Definition~\ref{defn:contour factor}).
\end{enumerate}
A \emph{radial meandering factor} of $f$ is a (piecewise-linear) map $s \colon [-1,1] \to [-1,1]$ which is \emph{sign preserving} (i.e.\ $s(x) \geq 0$ for all $x \geq 0$ and $s(x) \leq 0$ for all $x \leq 0$) and such that $f = t_f \circ s$.
\end{defn}

It follows from Proposition~\ref{prop:meandering factor} that there always exists at least one radial meandering factor of any given $f$, but it is not necessarily unique.  Observe that if $s$ is any radial meandering factor of $f$, then $s$ has no negative radial departures.  In fact, for any $y_1 < 0 < y_2$, there exists a positive radial departure $\langle x_1,x_2 \rangle$ of $s$ such that $s(x_1) = y_1$ and $s(x_2) = y_2$.

We develop several basic properties of radial contour factors and radial departures in the remainder of this section.

\begin{lem}
\label{lem:s match}
Let $f \colon [-1,1] \to [-1,1]$ be a piecewise-linear map with $f(0) = 0$, let $t_f$ be the radial contour factor of $f$, and let $\langle x_1,x_2 \rangle$ be a radial departure of $f$.  Then there exists a radial departure $\langle y_1,y_2 \rangle$ of $t_f$ such that for any map $s \colon [-1,1] \to [-1,1]$ with $s(0) = 0$ and $f = t_f \circ s$:
\begin{enumerate}
\item $\langle x_1,x_2 \rangle$ is a positive radial departure of $s$; and
\item $s(x_1) = y_1$ and $s(x_2) = y_2$.
\end{enumerate}
\end{lem}

\begin{proof}
First let $s_0$ be a radial meandering factor of $f$.  Since $f = t_f \circ s_0$, we have by Proposition~\ref{prop:comp dep} that $\langle x_1,x_2 \rangle$ is a positive radial departure of $s_0$ (positive because $s_0$ is a radial meandering factor), and $\langle s_0(x_1),s_0(x_2) \rangle$ is a radial departure of $t_f$ with the same orientation as $\langle x_1,x_2 \rangle$ has for $f$.  Let $y_1 = s_0(x_1)$ and $y_2 = s_0(x_2)$.

Now let $s \colon [-1,1] \to [-1,1]$ be any map with $s(0) = 0$ and $f = t_f \circ s$.  According to Proposition~\ref{prop:comp dep}, we have that either:
\begin{enumerate}[label=(\alph{*})]
\item \label{s option a} $\langle x_1,x_2 \rangle$ is a positive radial departure of $s$ and $\langle s(x_1),s(x_2) \rangle$ is a radial departure of $t_f$ with the same orientation as $\langle x_1,x_2 \rangle$ has for $f$; or
\item \label{s option b} $\langle x_1,x_2 \rangle$ is a negative radial departure of $s$ and $\langle s(x_2),s(x_1) \rangle$ is a radial departure of $t_f$ with the opposite orientation to what $\langle x_1,x_2 \rangle$ has for $f$.
\end{enumerate}
We claim that alternative \ref{s option b} is impossible.  Indeed, suppose for a contradiction that \ref{s option b} holds.  Then since $\{t_f(s_0(x_1)),t_f(s_0(x_2))\} = \{t_f(s(x_2)),t_f(s(x_1))\}$, we have a contradiction with Lemma~\ref{lem:dep values unique}(2), as $\langle s_0(x_1),s_0(x_2) \rangle$ and $\langle s(x_2),s(x_1) \rangle$ are radial departures of $t_f$ with opposite orientations.

Therefore \ref{s option a} holds, so $\langle x_1,x_2 \rangle$ is a positive radial departure of $s$.  Moreover, since $\{t_f(s_0(x_1)),t_f(s_0(x_2))\} = \{t_f(s(x_1)),t_f(s(x_2))\}$, we have by Lemma~\ref{lem:dep values unique}(1) that $s(x_1) = s_0(x_1) = y_1$ and $s(x_2) = s_0(x_2) = y_2$.
\end{proof}

\begin{defn}
\label{defn:same deps}
Let $f,g \colon [-1,1] \to [-1,1]$ be maps with $f(0) = g(0) = 0$.  We say \emph{$f$ and $g$ have the same radial departures} if for any $y_1,y_2 \in [-1,1]$, there exists a radial departure $\langle x_1,x_2 \rangle$ of $f$ with $f(x_1) = y_1$ and $f(x_2) = y_2$ if and only if there exists a radial departure $\langle x_1',x_2' \rangle$ of $g$ with $g(x_1') = y_1$ and $g(x_2') = y_2$.
\end{defn}

\begin{lem}
\label{lem:f tf same dep}
Let $f \colon [-1,1] \to [-1,1]$ be a piecewise-linear map with $f(0) = 0$, and let $t_f$ be the radial contour factor of $f$.  Then $f$ and $t_f$ have the same radial departures.
\end{lem}

\begin{proof}
Let $s$ be a radial meandering factor for $f$, and let $y_1,y_2 \in [-1,1]$.

Suppose there exists a radial departure $\langle x_1,x_2 \rangle$ of $f$ with $f(x_1) = y_1$ and $f(x_2) = y_2$.  Let $x_1' = s(x_1)$ and $x_2' = s(x_2)$.  Then by Proposition~\ref{prop:comp dep}, $\langle x_1,x_2 \rangle$ is a positive radial departure of $s$ (positive because $s$ is a radial meandering factor), and $\langle x_1',x_2' \rangle$ is a radial departure of $t_f$ with $t_f(x_1') = y_1$ and $t_f(x_2') = y_2$.

Conversely, suppose there exists a radial departure $\langle x_1',x_2' \rangle$ of $t_f$ with $f(x_1') = y_1$ and $f(x_2') = y_2$.  As observed after Definition~\ref{defn:rad contour factor}, there exists a positive radial departure $\langle x_1,x_2 \rangle$ of $s$ such that $s(x_1) = x_1'$ and $s(x_2) = x_2'$.  Then by Proposition~\ref{prop:comp dep}, $\langle x_1,x_2 \rangle$ is a radial departure of $f$, with $f(x_1) = y_1$ and $f(x_2) = y_2$.
\end{proof}

\begin{cor}
\label{cor:same contour same dep}
Let $f,g \colon [-1,1] \to [-1,1]$ be piecewise-linear maps with $f(0) = g(0) = 0$.  If $t_f = t_g$, then $f$ and $g$ have the same radial departures.
\end{cor}

\begin{proof}
This follows immediately from Lemma~\ref{lem:f tf same dep}.
\end{proof}

The converse of Corollary~\ref{cor:same contour same dep} does not hold in general.  See Figure~\ref{fig:same deps} for an example.

\begin{figure}
\begin{center}
\includegraphics{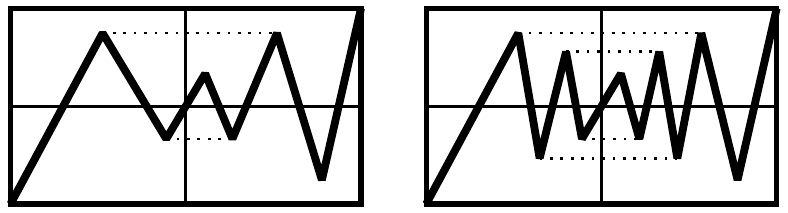}
\end{center}

\caption{An example of two maps which have the same radial departures, but do not have the same radial contour factors.}
\label{fig:same deps}
\end{figure}

\begin{prop}
\label{prop:comp same dep}
Let $f_1,f_2 \colon [-1,1] \to [-1,1]$ be maps with $f_1(0) = f_2(0) = 0$, and suppose $f_1$ and $f_1 \circ f_2$ have the same radial departures.  Then
\begin{enumerate}
\item For any radial departure $\langle y_1,y_2 \rangle$ of $f_1$, there exists a positive radial departure $\langle x_1,x_2 \rangle$ of $f_2$ with $f_2(x_1) = y_1$ and $f_2(x_2) = y_2$; and
\item For any negative radial departure $\langle x_1,x_2 \rangle$ of $f_2$, $\langle f_2(x_2),f_2(x_1) \rangle$ is not a radial departure of $f_1$.  In fact, if $\langle y_1,y_2 \rangle$ is any radial departure of $f_1$ then either
\[ f_2(x_2) < y_1 < 0 < y_2 < f_2(x_1) \quad \textrm{or} \quad y_1 < f_2(x_2) < 0 < f_2(x_1) < y_2 .\]
\end{enumerate}
\end{prop}

\begin{proof}
For (1), let $\langle y_1,y_2 \rangle$ be a radial departure of $f_1$.  Since $f_1$ and $f_1 \circ f_2$ have the same radial departures, there is a radial departure $\langle x_1,x_2 \rangle$ of $f_1 \circ f_2$ with $f_1 \circ f_2(x_1) = f_1(y_1)$ and $f_1 \circ f_2(x_2) = f_1(y_2)$.  By Proposition~\ref{prop:comp dep}, we have that $\langle f_2(x_1),f_2(x_2) \rangle$ is a radial departure of $f_1$, and, by Lemma~\ref{lem:dep values unique}(1), we deduce that $f_2(x_1) = y_1$ and $f_2(x_2) = y_2$.  Then by Proposition~\ref{prop:comp dep} again, we conclude that $\langle x_1,x_2 \rangle$ is a positive radial departure of $f_2$.

For (2), let $\langle x_1,x_2 \rangle$ be a negative radial departure of $f_2$, and let $\langle y_1,y_2 \rangle$ be any radial departure of $f_1$.  By part (1), there exists a positive radial departure $\langle x_1',x_2' \rangle$ of $f_2$ such that $f_2(x_1') = y_1$ and $f_2(x_2') = y_2$.  By Proposition~\ref{prop:alt dep nested}, we have either
\[ x_1 < x_1' < 0 < x_2' < x_2 \quad \textrm{or} \quad x_1' < x_1 < 0 < x_2 < x_2' \]
and correspondingly either
\[ f_2(x_2) < f_2(x_1') < 0 < f_2(x_2') < f_2(x_1) \quad \textrm{or} \quad f_2(x_1') < f_2(x_2) < 0 < f_2(x_1) < f_2(x_2') \]
which proves (2).
\end{proof}

\section{Contour twins}
\label{sec:no twins}

This section contains the main result of this paper, Theorem~\ref{thm:no twins}, which gives a partial affirmative answer to the question of Nadler and Quinn (more precisely to the form of this question we state as Question~\ref{ques:0 accessible} in Section~\ref{sec:prelim}), for inverse systems $\left \langle [-1,1],f_n \right \rangle$ satisfying two technical conditions.  The first condition is that the radial contour factors of the bonding maps are ``stable'', by which we mean that the radial contour factor of $f_n \circ f_{n+1}$ equals the radial contour factor of $f_n$, for each $n$.  The second condition states that the left and right ``halves'' of each bonding map $f_n$ must be sufficiently ``misaligned'', in the sense of the following Definition.

\begin{defn}
Let $f \colon [-1,1] \to [-1,1]$ be a map with $f(0) = 0$.  We say $f$ has \emph{no contour twins} if:
\begin{enumerate}
\item $0$ is neither a local minimum nor a local maximum of $f$; and
\item If $\alpha$ is a right contour point of $f$ and $\alpha'$ is a left contour point of $f$ with $0 < \alpha < 1$ and $-1 < \alpha' < 0$, then $f(\alpha) \neq f(\alpha')$.
\end{enumerate}
\end{defn}

\begin{lem}
\label{lem:sell}
Let $f_1,f_2,f_3 \colon [-1,1] \to [-1,1]$ be piecewise-linear maps with $f_1(0) = f_2(0) = f_3(0) = 0$, and such that:
\begin{enumerate}[label=(\roman{*})]
\item $t_{f_1} = t_{f_1 \circ f_2}$;
\item $t_{f_2} = t_{f_2 \circ f_3}$; and
\item each of $f_1$ and $f_2$ has no contour twins.
\end{enumerate}
Then there exists $\tilde{s} \colon [-1,1] \to [-1,1]$ with $\tilde{s}(0) = 0$ such that:
\begin{enumerate}
\item $t_{f_1} \circ \tilde{s} = f_1 \circ f_2$; and
\item $\tilde{s} \circ t_{f_3}$ has no negative radial departures.
\end{enumerate}
\end{lem}

\begin{proof}
For $k = 1,2,3$, let $t_k = t_{f_k}$.  Let $s_1$ be a radial meandering factor of $f_1$ (so that $f_1 = t_1 \circ s_1$) and let $s_{12}$ be a radial meandering factor of $f_1 \circ f_2$ (so that $f_1 \circ f_2 = t_1 \circ s_{12}$).  Define $\tilde{s} \colon [-1,1] \to [-1,1]$ by:
\[ \tilde{s}(x) = \begin{cases}
s_{12}(x) & \textrm{if } x \leq 0 \\
s_1 \circ f_2(x) & \textrm{if } x > 0 .
\end{cases} \]

See Example~\ref{ex:minc demo} below for a particular example with this map $\tilde{s}$ shown.

Note that $\tilde{s}$ has no negative radial departures, so by Proposition~\ref{prop:comp dep} we need only prove that for any negative radial departure $\langle w_1,w_2 \rangle$ of $t_3$, $\langle t_3(w_2),t_3(w_1) \rangle$ is not a positive radial departure of $\tilde{s}$.

Let $\langle w_1,w_2 \rangle$ be a negative radial departure of $t_3$, and let $x_1 = t_3(w_2)$ and $x_2 = t_3(w_1)$.  Suppose for a contradiction that $\langle x_1,x_2 \rangle$ is a positive radial departure of $\tilde{s}$.  This means
\[ s_{12}((x_1,0]) = (s_{12}(x_1),0] \quad \textrm{and} \quad s_1 \circ f_2([0,x_2)) \subseteq (s_{12}(x_1),s_1 \circ f_2(x_2)) .\]

The remainder of the proof proceeds in two parts.  In the first part, we establish the existence and configuration of a few departures of $s_1 \circ f_2$, using the assumption that $f_2$ has no contour twins, and the observation \ref{properly nested} below.  In the second part, we consider the map $t_1$, and the relationship between $s_1 \circ f_2$ and $s_{12}$, and derive a contradiction.

\medskip
\noindent \textbf{Part 1.}

Recall that, by Proposition~\ref{prop:comp dep}, each radial departure of $s_1 \circ f_2$ is also a radial departure of $f_2$.  Further, it follows from Lemma~\ref{lem:same comp contour} that $f_2 \circ f_3$ and $f_2 \circ t_3$ have the same radial contour factor, so $f_2$ and $f_2 \circ t_3$ have the same radial contour factor and hence the same radial departures by Corollary~\ref{cor:same contour same dep}.  Proposition~\ref{prop:comp same dep}(2) then implies that
\begin{enumerate}[label=($\ast$), ref=($\ast$)]
\item \label{properly nested} If $\langle z_1,z_2 \rangle$ is a radial departure of $s_1 \circ f_2$, then either $z_1 < x_1 < 0 < x_2 < z_2$ or $x_1 < z_1 < 0 < z_2 < x_2$
\end{enumerate}
since any such $\langle z_1,z_2 \rangle$ is a radial departure of $f_2$ by Proposition~\ref{prop:comp dep}.

Since $0$ is not a local minimum of $f_2$ and $s_1$ is sign-preserving, it follows that $0$ is also not a local minimum of $s_1 \circ f_2$.  We claim there exists $x \in (0,x_2)$ such that $s_1 \circ f_2(x) < 0$.  Indeed, if not, then since $0$ is not a local minimum of $s_1 \circ f_2$, for some arbitrarily small $\varepsilon > 0$ we would have that $\langle -\varepsilon,x_2 \rangle$ is a (positive) radial departure of $s_1 \circ f_2$, but this would contradict \ref{properly nested}.

Let $y_0 = \min s_1 \circ f_2([0,x_2])$, and let $x_0 \in (0,x_2)$ be the right departure of $s_1 \circ f_2$ such that $s_1 \circ f_2(x_0) = y_0$.  Also, let $y_2 = s_1 \circ f_2(x_2)$.  Notice that $x_0$ is a contour point of $s_1 \circ f_2$.

Because $t_3([w_1,0]) \supseteq [0,x_2]$, we have that $y_0,y_2 \in s_1 \circ f_2 \circ t_3([w_1,0])$, therefore there exist left departures $\tilde{w}_0,\tilde{w}_2 < 0$ of $s_1 \circ f_2 \circ t_3$ such that $s_1 \circ f_2 \circ t_3(\tilde{w}_0) = y_0$ and $s_1 \circ f_2 \circ t_3(\tilde{w}_2) = y_2$.  Since $s_1 \circ f_2 \circ t_3$ and $s_1 \circ f_2$ have the same radial contour factor (by Lemma~\ref{lem:same comp contour}), there exist corresponding left departures $\tilde{x}_0,\tilde{x}_2 < 0$ of $s_1 \circ f_2$ such that $s_1 \circ f_2(\tilde{x}_0) = y_0$ and $s_1 \circ f_2(\tilde{x}_2) = y_2$.

\begin{claim}
\label{claim:dep order}
$\tilde{x}_0 < \tilde{x}_2$ and $x_1 < \tilde{x}_2$.
\end{claim}

\begin{proof}[Proof of Claim~\ref{claim:dep order}]
\renewcommand{\qedsymbol}{\textsquare (Claim~\ref{claim:dep order})}
Suppose for a contradiction that $\tilde{x}_2 < \tilde{x}_0$.  First note that $\tilde{x}_0$ cannot be a contour point of $s_1 \circ f_2$, because if it were then $\tilde{x}_0$ and $x_0$ would be contour twins of $s_1 \circ f_2$, and hence of $f_2$, contradicting the hypotheses of the Lemma.  Therefore there must be a left departure $x$ of $s_1 \circ f_2$ with $\tilde{x}_2 < x < \tilde{x}_0$ and $s_1 \circ f_2(x) < s_1 \circ f_2(\tilde{x}_0)$.  But then $\langle x,x_2 \rangle$ is a (positive) radial departure of $s_1 \circ f_2$, contradicting \ref{properly nested}.  This proves that $\tilde{x}_0 < \tilde{x}_2$.

We now have that $\langle \tilde{x}_2,x_0 \rangle$ is a negative radial departure of $s_1 \circ f_2$.  Therefore by \ref{properly nested}, we must have $x_1 < \tilde{x}_2$.
\end{proof}

See Figure~\ref{fig:part1} for an illustration summarizing what has been deduced so far in Part 1.

\begin{figure}
\begin{center}
\includegraphics{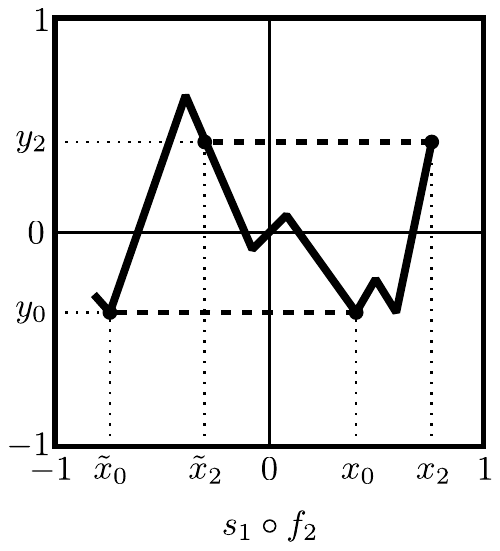}
\end{center}

\caption{A sample illustration of part of the map $s_1 \circ f_2$, with the departures $x_0,x_2,\tilde{x}_0,\tilde{x}_2$ shown.}
\label{fig:part1}
\end{figure}

\begin{claim}
\label{claim:above y0}
For all $x \in [x_1,0]$, $s_1 \circ f_2(x) \geq y_0$.
\end{claim}

\begin{proof}[Proof of Claim~\ref{claim:above y0}]
\renewcommand{\qedsymbol}{\textsquare (Claim~\ref{claim:above y0})}
Suppose for a contradiction that there exists $p_0 \in [x_1,0]$ such that $s_1 \circ f_2(p_0) < y_0$.  This means $x_1 \leq p_0 < \tilde{x}_0$.  We may assume that $p_0$ is a left departure of $s_1 \circ f_2$.  Let $q_3 = s_1 \circ f_2(p_0)$.

Because $t_3([0,w_2]) \supseteq [x_1,0]$, we have that $q_3 \in s_1 \circ f_2 \circ t_3([0,w_2])$, therefore there exists a right departure $\tilde{w}_3$ of $s_1 \circ f_2 \circ t_3$ such that $s_1 \circ f_2 \circ t_3(\tilde{w}_3) = q_3$.  Since $s_1 \circ f_2 \circ t_3$ and $s_1 \circ f_2$ have the same radial contour factor, there exists a corresponding right departure $\tilde{q}_0$ of $s_1 \circ f_2$ such that $s_1 \circ f_2(\tilde{q}_0) = q_3$.  Clearly $\tilde{q}_0 > x_2$.

Let $q_1 = \max \{s_1 \circ f_2(x): x \in [p_0,\tilde{x}_2]\}$, let $q_2 = \max \{s_1 \circ f_2(x): x \in [x_2,\tilde{q}_0]\}$, and let $p_1 \in [p_0,\tilde{x}_2]$ be the left departure of $s_1 \circ f_2$ and let $p_2 \in [x_2,\tilde{q}_0]$ be the right departure of $s_1 \circ f_2$ such that $s_1 \circ f_2(p_1) = q_1$ and $s_1 \circ f_2(p_2) = q_2$.

Observe that both $p_1$ and $p_2$ are contour points of $s_1 \circ f_2$, hence $q_1 \neq q_2$ according to the hypotheses of the Lemma.  But if $q_1 > q_2$ then $\langle p_1,\tilde{q}_0 \rangle$ is a radial departure of $s_1 \circ f_2$ contradicting \ref{properly nested}, and if $q_2 > q_1$ then $\langle p_0,p_2 \rangle$ is a radial departure of $s_1 \circ f_2$ contradicting \ref{properly nested}.
\end{proof}

\medskip
\noindent \textbf{Part 2.}

Suppose $f_1$ has $n$ right contour points (not including $0$), which means the non-zero right contour points of $t_1$ are $\frac{1}{n},\frac{2}{n},\ldots,1$.

\begin{claim}
\label{claim:y2 past contour pt}
$y_2 > \frac{1}{n}$.
\end{claim}

\begin{proof}[Proof of Claim~\ref{claim:y2 past contour pt}]
\renewcommand{\qedsymbol}{\textsquare (Claim~\ref{claim:y2 past contour pt})}
Suppose $f_1$ has $m$ left contour points (not including $0$), which means the non-zero left contour points of $t_1$ are $-1,\ldots,\frac{-2}{m},\frac{-1}{m}$.  Since $0$ is not a local minimum or a local maximum for $f_1$, and hence for $t_1$, it follows that $t_1$ is one-to-one on $\left[ \frac{-1}{m},\frac{1}{n} \right]$.  Since $s_{12}(0) = s_1 \circ f_2(0) = 0$, and $t_1 \circ s_{12} = t_1 \circ s_1 \circ f_2$, it must be the case that $s_{12}(x) = s_1 \circ f_2(x)$ for all $x > 0$ for which $s_1 \circ f_2([0,x]) \subseteq \left( \frac{-1}{m},\frac{1}{n} \right)$.

Suppose for a contradiction that $y_2 \leq \frac{1}{n}$.  Now $s_{12}(x) \geq 0 > \frac{-1}{m}$ for all $x \geq 0$, and $s_1 \circ f_2(x) < y_2 \leq \frac{1}{n}$ for all $x \in [0,x_0] \subset [0,x_2)$.  Together, these imply that $s_1 \circ f_2(x) = s_{12}(x)$ for all $x \in [0,x_0]$.  But $s_1 \circ f_2(x_0) = y_0 < 0$, while $s_{12}(x) \geq 0$, a contradiction.
\end{proof}

Let $i \in \{2,\ldots,n\}$ be such that $\frac{i-1}{n} < y_2 \leq \frac{i}{n}$.  For simplicity, we assume that $t_1 \left( \frac{i}{n} \right) < 0$, which means $t_1 \left( \frac{i-1}{n} \right) > 0$.  The case where $t_1 \left( \frac{i}{n} \right) > 0$ can be argued similarly.

\begin{claim}
\label{claim:y3}
There exists a left departure $y_3 \in (s_{12}(x_1),0]$ of $t_1$ such that $t_1(y_3) = t_1 \left( \frac{i-1}{n} \right)$.
\end{claim}

See Figure~\ref{fig:y3} for an illustration of the contour points $\frac{i-1}{n}$ and $\frac{i}{n}$ of $t_1$ together with $y_3$.

\begin{figure}
\begin{center}
\includegraphics{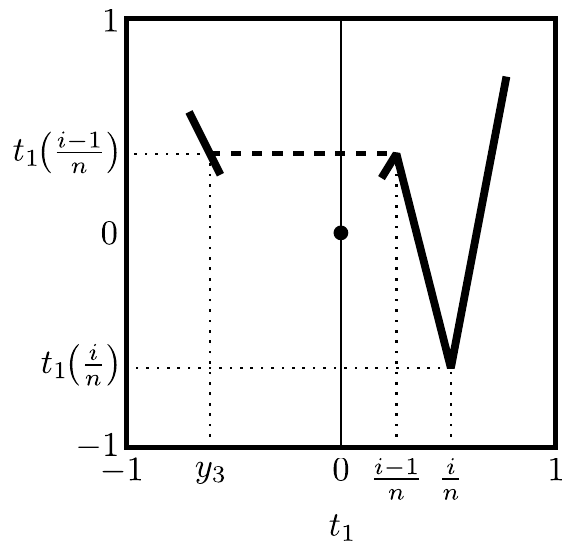}
\end{center}

\caption{A sample illustration of part of the map $t_1$, with the contour points $\frac{i-1}{n}$, $\frac{i}{n}$, and the point $y_3$ shown.}
\label{fig:y3}
\end{figure}

\begin{proof}[Proof of Claim~\ref{claim:y3}]
\renewcommand{\qedsymbol}{\textsquare (Claim~\ref{claim:y3})}
It suffices to prove that $t_1 \left( \tfrac{i-1}{n} \right) \in t_1((s_{12}(x_1),0])$.  Observe that $\frac{i-1}{n} \in [0,y_2) \subseteq s_1 \circ f_2((\tilde{x}_2,0])$, and so
\begin{align*}
t_1 \left( \tfrac{i-1}{n} \right) &\in t_1 \circ s_1 \circ f_2((\tilde{x}_2,0]) \\
&= t_1 \circ s_{12}((\tilde{x}_2,0]) \quad \textrm{since $t_1 \circ s_1 \circ f_2 = t_1 \circ s_{12}$} \\
&\subseteq t_1 \circ s_{12}((x_1,0]) \quad \textrm{by Claim~\ref{claim:dep order}} \\
&= t_1((s_{12}(x_1),0]) .  
\end{align*}
\end{proof}

Note that
\begin{enumerate}[label=($\dagger$), ref=($\dagger$)]
\item \label{y3 range} $t_1(y) \leq t_1 \left( \frac{i-1}{n} \right)$ for each $y \in \left[ y_3, \frac{i}{n} \right]$.
\end{enumerate}

\begin{claim}
\label{claim:y4}
There exists a left departure $y_4 \in (y_3,0]$ of $t_1$ such that $t_1(y_4) = t_1 \left( \frac{i}{n} \right)$.  Further, $i < n$.
\end{claim}

\begin{proof}[Proof of Claim~\ref{claim:y4}]
\renewcommand{\qedsymbol}{\textsquare (Claim~\ref{claim:y4})}
We need only argue that $t_1 \left( \frac{i}{n} \right) \in t_1([y_3,0])$, and that $i < n$.  We consider two cases below, depending on whether $y_3 \leq y_0$ or $y_0 < y_3$.  See Figure~\ref{fig:y4} for an illustration of the points constructed in the arguments below.

\begin{figure}
\begin{center}
\includegraphics{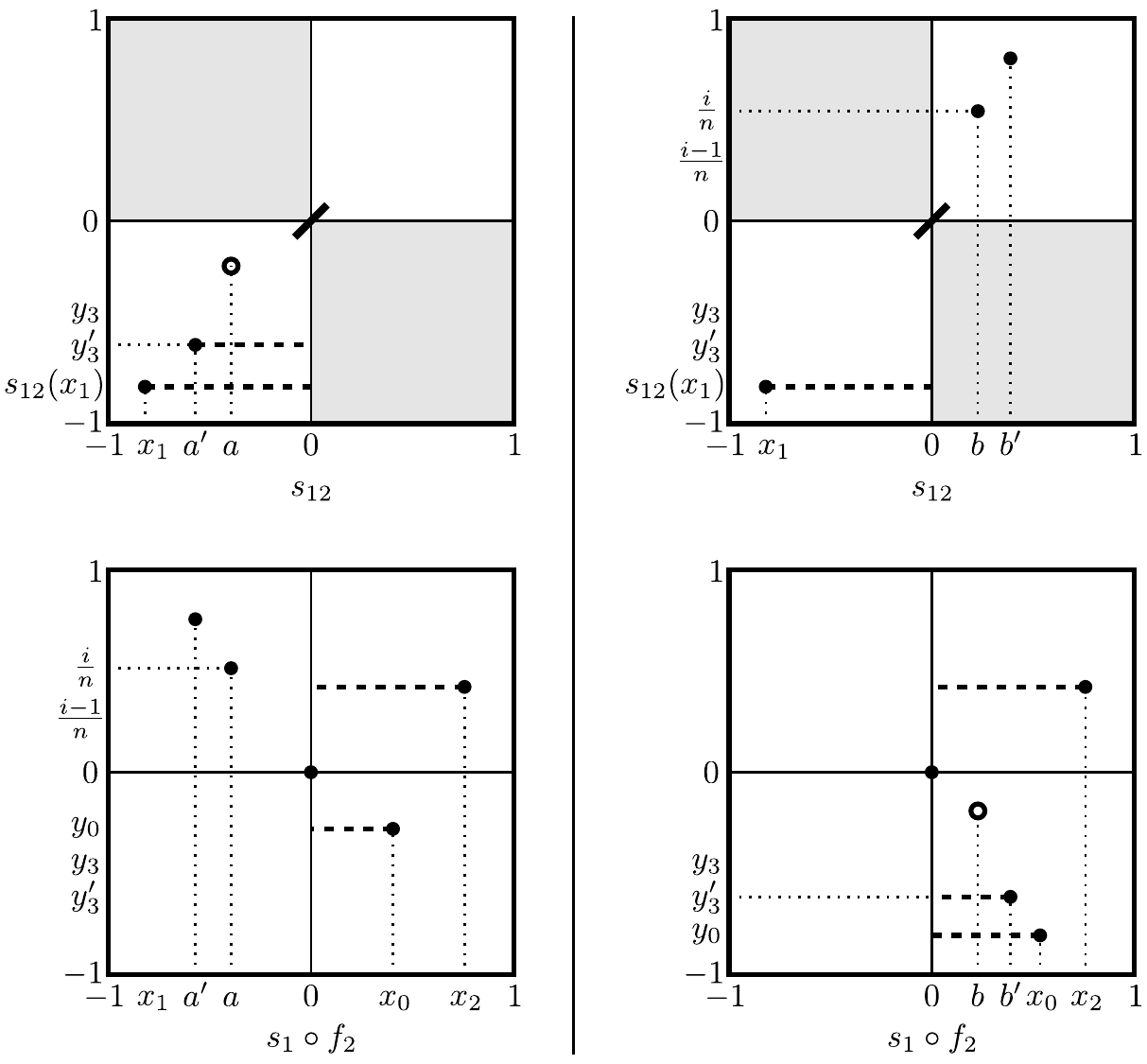}
\end{center}

\caption{A sample illustration of parts of the maps $s_{12}$ and $s_1 \circ f_2$, for Case~1 on the left and Case~2 on the right, in the proof of Claim~\ref{claim:y4}.  In each case, the hollow point witnesses that $t_1 \left( \frac{i}{n} \right) \in t_1([y_3,0])$.}
\label{fig:y4}
\end{figure}

Let $y_3'$ be such that $s_{12}(x_1) < y_3' < y_3$ and $t_1(y) > 0$ for each $y \in [y_3',y_3]$.  Since $t_1$ has no contour twins, $y_3$ is not a contour point of $t_1$, so we may assume that $y_3'$ is a left departure of $t_1$, so that $t_1(y_3') > t_1 \left( \frac{i-1}{n} \right)$.

\medskip
\textbf{Case~1:} Suppose $y_3 \leq y_0$.

Let $a'$ be a left departure of $s_{12}$ such that $s_{12}(a') = y_3'$, so that $x_1 < a'$.  Since $s_1 \circ f_2(a') \geq y_0 \geq y_3$ (by Claim~\ref{claim:above y0}) and $t_1 \circ s_1 \circ f_2(a') = t_1 \circ s_{12}(a') = t_1(y_3') > t_1 \left( \frac{i-1}{n} \right)$, it follows from \ref{y3 range} that $i < n$ and $s_1 \circ f_2(a') > \frac{i}{n}$.  Thus there exists $a \in (a',0]$ such that $s_1 \circ f_2(a) = \frac{i}{n}$.  Observe that $s_{12}(a)$ cannot be in $[y_3',y_3]$ because $t_1 \circ s_{12}(a) = t_1 \circ s_1 \circ f_2(a) = t_1 \left( \frac{i}{n} \right) < 0$.  Therefore $s_{12}(a) \in (y_3,0]$, hence $t_1 \left( \frac{i}{n} \right) = t_1 \circ s_{12}(a) \in t_1((y_3,0])$.

\medskip
\textbf{Case~2:} Suppose $y_0 < y_3$.

Here we may assume that $y_0 < y_3'$ as well.  Let $b'$ be a right departure of $s_1 \circ f_2$ such that $s_1 \circ f_2(b') = y_3'$, so that $b' < x_0$.  Then $t_1 \circ s_{12}(b') = t_1 \circ s_1 \circ f_2(b') = t_1(y_3') > t_1 \left( \frac{i-1}{n} \right)$.  It follows from \ref{y3 range} that $i < n$ and $s_{12}(b') > \frac{i}{n}$.  Thus there exists $b \in [0,b')$ such that $s_{12}(b) = \frac{i}{n}$.  Observe that $s_1 \circ f_2(b)$ cannot be in $[y_3',y_3]$ because $t_1 \circ s_1 \circ f_2(b) = t_1 \circ s_{12}(b) = t_1 \left( \frac{i}{n} \right) < 0$.  Also, $s_1 \circ f_2(b)$ cannot be greater than $0$, because $s_1 \circ f_2(b) < y_2 \leq \frac{i}{n}$.  Therefore $s_1 \circ f_2(b) \in (y_3,0]$, hence $t_1 \left( \frac{i}{n} \right) = t_1 \circ s_1 \circ f_2(b) \in t_1((y_3,0])$.
\end{proof}

Note that
\begin{enumerate}[label=($\dagger\dagger$), ref=($\dagger\dagger$)]
\item \label{y4 range} $t_1(y) \geq t_1 \left( \frac{i}{n} \right)$ for each $y \in \left[ y_4, \frac{i+1}{n} \right]$.
\end{enumerate}

The next Claim immediately leads to a contradiction with \ref{y3 range}, and as such will complete the proof of Lemma~\ref{lem:sell}.  The proof is very similar to the proof of Claim~\ref{claim:y4}, but we include it for completeness.

\begin{claim}
\label{claim:y5}
There exists $y_5 \in (y_4,0]$ such that $t_1(y_5) = t_1 \left( \frac{i+1}{n} \right)$.
\end{claim}

\begin{proof}[Proof of Claim~\ref{claim:y5}]
\renewcommand{\qedsymbol}{\textsquare (Claim~\ref{claim:y5})}
To argue that $t_1 \left( \frac{i+1}{n} \right) \in t_1((y_4,0])$, we consider two cases below, depending on whether $y_4 \leq y_0$ or $y_0 < y_4$.

Let $y_4'$ be such that $y_3 < y_4' < y_4$ and $t_1(y) < 0$ for each $y \in [y_4',y_4]$.  Since $t_1$ has no contour twins, $y_4$ is not a contour point of $t_1$, so we may assume that $y_4'$ is a left departure of $t_1$, so that $t_1(y_4') < t_1 \left( \frac{i}{n} \right)$.

\medskip
\textbf{Case~1:} Suppose $y_4 \leq y_0$.

Let $a'$ be a left departure of $s_{12}$ such that $s_{12}(a') = y_4'$, so that $x_1 < a'$.  Since $s_1 \circ f_2(a') \geq y_0 \geq y_4$ (by Claim~\ref{claim:above y0}) and $t_1 \circ s_1 \circ f_2(a') = t_1 \circ s_{12}(a') = t_1(y_4') < t_1 \left( \frac{i}{n} \right)$, it follows from \ref{y4 range} that $i+1 < n$ and $s_1 \circ f_2(a') > \frac{i+1}{n}$.  Thus there exists $a \in (a',0]$ such that $s_1 \circ f_2(a) = \frac{i+1}{n}$.  Observe that $s_{12}(a)$ cannot be in $[y_4',y_4]$ because $t_1 \circ s_{12}(a) = t_1 \circ s_1 \circ f_2(a) = t_1 \left( \frac{i+1}{n} \right) > 0$.  Therefore $s_{12}(a) \in (y_4,0]$, hence $t_1 \left( \frac{i+1}{n} \right) = t_1 \circ s_{12}(a) \in t_1((y_4,0])$.

\medskip
\textbf{Case~2:} Suppose $y_0 < y_4$.

Here we may assume that $y_0 < y_4'$ as well.  Let $b'$ be a right departure of $s_1 \circ f_2$ such that $s_1 \circ f_2(b') = y_4'$, so that $b' < x_0$.  Then $t_1 \circ s_{12}(b') = t_1 \circ s_1 \circ f_2(b') = t_1(y_4') > t_1 \left( \frac{i}{n} \right)$.  It follows from \ref{y4 range} that $i+1 < n$ and $s_{12}(b') > \frac{i+1}{n}$.  Thus there exists $b \in [0,b')$ such that $s_{12}(b) = \frac{i+1}{n}$.  Observe that $s_1 \circ f_2(b)$ cannot be in $[y_4',y_4]$ because $t_1 \circ s_1 \circ f_2(b) = t_1 \circ s_{12}(b) = t_1 \left( \frac{i+1}{n} \right) > 0$.  Also, $s_1 \circ f_2(b)$ cannot be greater than $0$, because $s_1 \circ f_2(b) < y_2 \leq \frac{i}{n} < \frac{i+1}{n}$.  Therefore $s_1 \circ f_2(b) \in (y_4,0]$, hence $t_1 \left( \frac{i+1}{n} \right) = t_1 \circ s_1 \circ f_2(b) \in t_1((y_4,0])$.
\end{proof}

We now have a contradiction with \ref{y3 range}, since $y_5 \in (y_4,0] \subset \left[ y_3,\frac{i}{n} \right]$ and $t_1(y_5) = t_1 \left( \frac{i+1}{n} \right) > t_1 \left( \frac{i-1}{n} \right)$.  This completes the proof of Lemma~\ref{lem:sell}.
\end{proof}

\begin{thm}
\label{thm:no twins}
Let $f_n \colon [-1,1] \to [-1,1]$, $n = 1,2,\ldots$, be piecewise-linear maps with $f_n(0) = 0$ for each $n$.  Suppose that for each $n$,
\begin{enumerate}
\item $f_n$ and $f_n \circ f_{n+1}$ have the same radial contour factor; and
\item $f_n$ has no contour twins.
\end{enumerate}
Then there exists an embedding of $X = \varprojlim \left \langle [-1,1], f_n \right \rangle$ into $\mathbb{R}^2$ for which the point $\langle 0,0,\ldots \rangle \in X$ is accessible.
\end{thm}

\begin{proof}
For each odd $n \geq 1$, let $t_n$ be the radial contour factor of $f_n$, and apply Lemma~\ref{lem:sell} to obtain a map $\tilde{s}_n \colon [-1,1] \to [-1,1]$ with $\tilde{s}_n(0) = 0$ such that $t_n \circ \tilde{s}_n = f_n \circ f_{n+1}$ and $\tilde{s}_n \circ t_{n+2}$ has no negative radial departures.  Then by Proposition~\ref{prop:embed 0 accessible} there exists an embedding of $\displaystyle \varprojlim_{n \textrm{ odd}} \left \langle [-1,1], \tilde{s}_n \circ t_{n+2} \right \rangle$ into $\mathbb{R}^2$ for which the point $\langle 0,0,\ldots \rangle$ is accessible.

Also, by standard properties of inverse limits (see Section~\ref{sec:prelim}),
\begin{align*}
X = \varprojlim \left \langle [-1,1],f_n \right \rangle &\approx \varprojlim \left \langle [-1,1],f_n \circ f_{n+1} \right \rangle_{n \textrm{ odd}} \quad \textrm{by composing bonding maps} \\
&= \varprojlim \left \langle [-1,1],t_n \circ \tilde{s}_n \right \rangle_{n \textrm{ odd}} \\
&\approx \varprojlim \left \langle [-1,1], t_1, [-1,1], \tilde{s}_1, [-1,1], t_3, [-1,1], \tilde{s}_3, \ldots \right \rangle \\ & \hspace{1.5in} \textrm{by (un)composing bonding maps} \\
&\approx \varprojlim \left \langle [-1,1], \tilde{s}_1, [-1,1], t_3, [-1,1], \tilde{s}_3, [-1,1], t_5, \ldots \right \rangle \\ & \hspace{1.5in} \textrm{by dropping first coordinate} \\
&\approx \varprojlim \left \langle [-1,1],\tilde{s}_n \circ t_{n+2} \right \rangle_{n \textrm{ odd}} \quad \textrm{by composing bonding maps.}
\end{align*}
In fact, an explicit homeomorphism between $X$ and $\displaystyle \varprojlim_{n \textrm{ odd}} \left \langle [-1,1],\tilde{s}_n \circ t_{n+2} \right \rangle$ is given by
\[ h \left( \langle x_n \rangle_{n=1}^\infty \right) = \langle \tilde{s}_{2k-1}(x_{2k+1}) \rangle_{k=1}^\infty .\]  Clearly $h(\langle 0,0,\ldots \rangle) = \langle 0,0,\ldots \rangle$.  Thus there exists an embedding of $X$ into $\mathbb{R}^2$ for which the point $\langle 0,0,\ldots \rangle \in X$ is accessible.
\end{proof}

We present two examples to complement the proof of Lemma~\ref{lem:sell} above.  In the first, we demonstrate the function $\tilde{s}$ constructed in the proof of Lemma~\ref{lem:sell} when each of the maps $f_1,f_2,f_3$ is Minc's map $f_M$.  In the second, we substantiate the need for the assumption that the maps $f_1,f_2,f_3$ have no contour twins, by showing an instance where $f_2$ has contour twins and the map $\tilde{s}$ as constructed in the proof of Lemma~\ref{lem:sell} does not satisfy the conclusion of the Lemma.

\begin{example}
\label{ex:minc demo}
Let $f_1 = f_2 = f_3 = f_M$, where $f_M$ is the map of Minc described in the Introduction (Figure~\ref{fig:minc}), rescaled so as to be a map on $[-1,1]$ instead of $[0,1]$.  It can be seen that $f_M^2$ has the same radial contour factor as $f_M$.  In Figure~\ref{fig:minc demo} we show the maps $f_M$, $f_M^2$, $t_{f_M}$, $s_1$ (such that $f_M = t_{f_M} \circ s_1$), $s_{12}$ (such that $f_M^2 = t_{f_M} \circ s_{12}$), $\tilde{s}$ (as defined in the proof of Lemma~\ref{lem:sell}), and finally $\tilde{s} \circ t_{f_M}$, which has no negative radial departures (as implied by the proof of Lemma~\ref{lem:sell}).

\begin{figure}
\begin{center}
\includegraphics{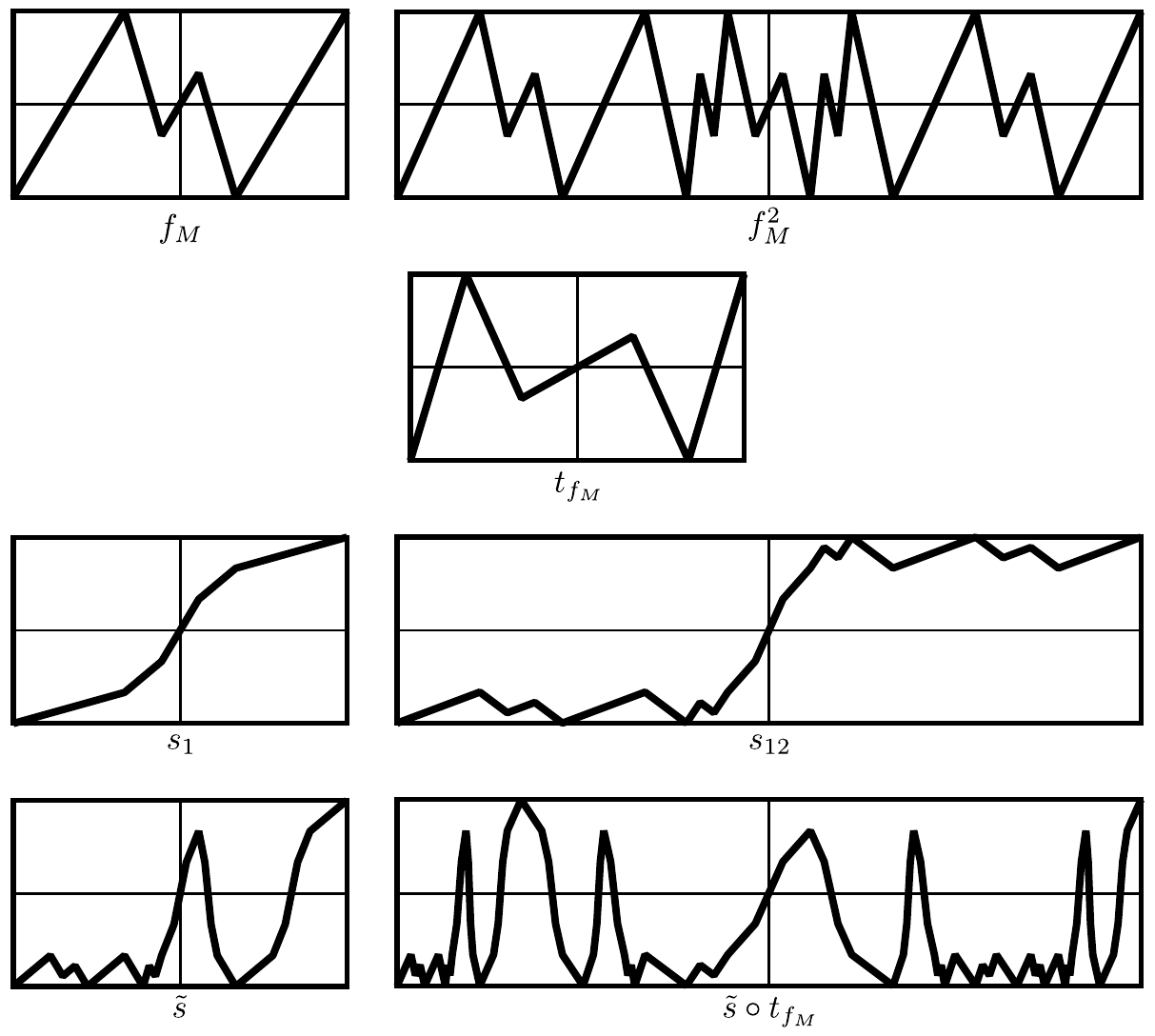}
\end{center}

\caption{The maps $f_M$, $f_M^2$, $t_{f_M}$, $s_1$, $s_{12}$, $\tilde{s}$, and $\tilde{s} \circ t_{f_M}$, for Example~\ref{ex:minc demo}.}
\label{fig:minc demo}
\end{figure}
\end{example}

\begin{example}
\label{ex:twins ex}
In Figure~\ref{fig:twins ex} we present an example of functions $f_1,f_2,f_3$ where $f_2$ has contour twins, and three natural choices of the map $\tilde{s}$ do not satisfy the conclusion of Lemma~\ref{lem:sell}.  We briefly list the key properties for this example below:
\begin{itemize}
\item $f_1 = t_1$, $f_2 = t_2$, $f_3 = t_3$.  Consequently, $s_1 = \mathrm{id}$.
\item $t_{f_1 \circ f_2} = t_1$, and $t_{f_2 \circ f_3} = t_2$.
\item $f_2$ has contour twins: namely, $f(b_3) = f(b_4)$ and $f(b_1) = f(b_5)$.
\item There is a unique radial meandering factor $s_{12} \colon [-1,1] \to [-1,1]$ of $f_1 \circ f_2$, given in red in Figure~\ref{fig:twins ex}.
\item There is a left departure $x_1 \in (b_2,b_3)$ of $s_{12}$ such that $s_{12}(x_1) = -1$.  There is a right departure $x_2 \in (b_4,b_5)$ of $s_{12}$ such that $s_{12}(x_2) = 1$.
\item Let $w_1 \in (a_1,a_2)$ be the left departure of $f_3$ such that $f_3(w_1) = x_2$, and let $w_2 \in (a_3,a_4)$ be the right departure of $f_3$ such that $f_3(w_2) = x_1$.  Then:
\begin{enumerate}[label=(\roman{*})]
\item $\langle w_1,w_2 \rangle$ is a negative radial departure of $s_{12} \circ t_3$;
\item $\langle a_1,w_2 \rangle$ is a negative radial departure of $\tilde{s} \circ t_3$, where
\[ \tilde{s}(x) = \begin{cases}
s_{12}(x) & \textrm{if } x \leq 0 \\
s_1 \circ f_2(x) & \textrm{if } x > 0
\end{cases} \]
as in the proof of Lemma~\ref{lem:sell}; and
\item $\langle w_1,a_4 \rangle$ is a negative radial departure of $\tilde{s} \circ t_3$, where
\[ \tilde{s}(x) = \begin{cases}
s_1 \circ f_2(x) & \textrm{if } x \leq 0 \\
s_{12}(x) & \textrm{if } x > 0 ,
\end{cases} \]
a variant of the map constructed in the proof of Lemma~\ref{lem:sell}.
\end{enumerate}
\end{itemize}
Nevertheless, it is not difficult to construct a different map $\tilde{s}$ for this example satisfying the conclusion of Lemma~\ref{lem:sell}, for example:
\[ \tilde{s}(x) = \begin{cases}
s_1 \circ f_2(x) & \textrm{if } x \in [b_1,b_6] \\
s_{12}(x) & \textrm{otherwise.}
\end{cases} \]

\begin{figure}
\begin{center}
\includegraphics{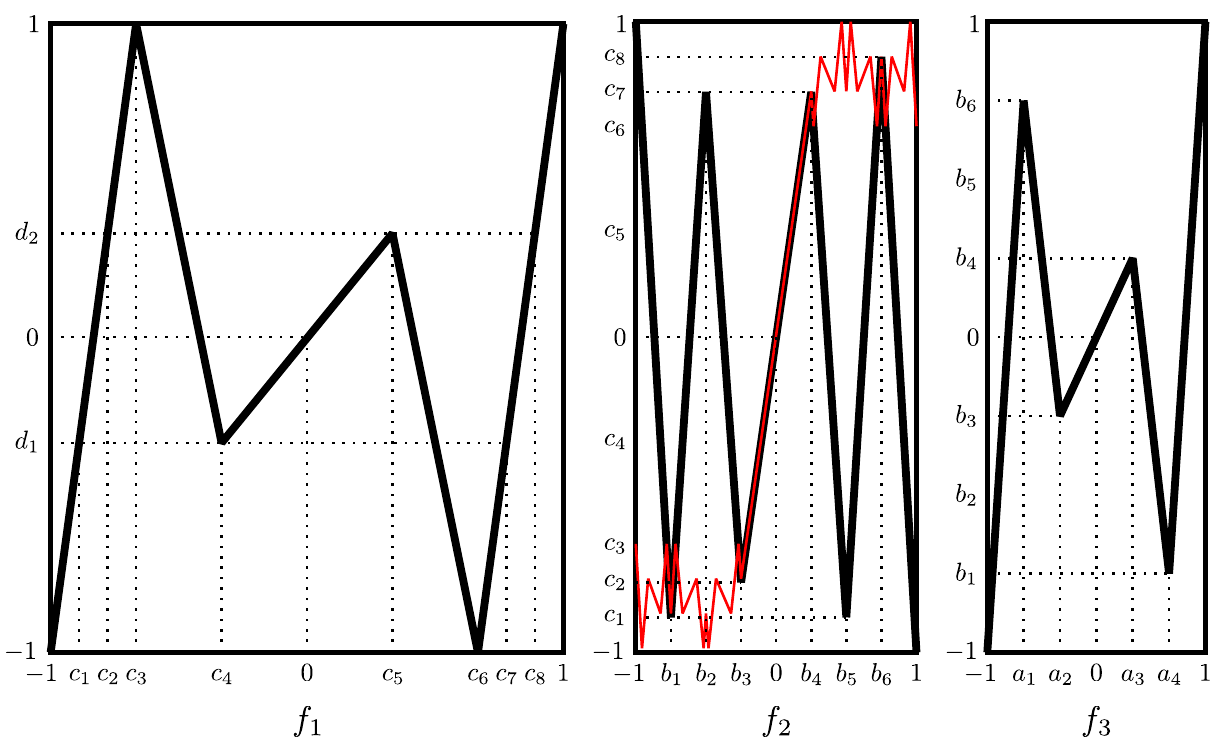}
\end{center}

\caption{The maps $f_1$, $f_2$, $f_3$ for Example~\ref{ex:twins ex}, as well as the map $s_{12}$ (in red).  Here $\{a_1,\ldots,a_4\} = \{\frac{-2}{3},\frac{-1}{3},\frac{1}{3},\frac{2}{3}\}$, $\{b_1,\ldots,b_6\} = \{\frac{-3}{4},\frac{-1}{2},\frac{-1}{4},\frac{1}{4},\frac{1}{2},\frac{3}{4}\}$, $\{c_1,\ldots,c_8\} = \{\frac{-8}{9},\frac{-7}{9},\frac{-2}{3},\frac{-1}{3},\frac{1}{3},\frac{2}{3},\frac{7}{9},\frac{8}{9}\}$, and $\{d_1,d_2\} = \{\frac{-1}{3},\frac{1}{3}\}$.}
\label{fig:twins ex}
\end{figure}
\end{example}

It would be interesting to know whether the assumption that the maps $f_n$ have no contour twins in Lemma~\ref{lem:sell} (and hence in Theorem~\ref{thm:no twins}) can be weakened or removed:

\begin{question}
\label{ques:no no twins}
Let $f_1,f_2,f_3 \colon [-1,1] \to [-1,1]$ be piecewise-linear maps with $f_1(0) = f_2(0) = f_3(0) = 0$, and such that $t_{f_1} = t_{f_1 \circ f_2}$ and $t_{f_2} = t_{f_2 \circ f_3}$.  Does there exist a map $\tilde{s} \colon [-1,1] \to [-1,1]$ with $\tilde{s}(0) = 0$ such that $f_1 \circ f_2 = t_{f_1} \circ \tilde{s}$ and $\tilde{s} \circ t_{f_3}$ has no negative radial departures?
\end{question}

We point out that if the answer to Question~\ref{ques:no no twins} is affirmative, it would imply that for any simplicial inverse system $\langle [-1,1],f_n \rangle$, for any point $x \in X = \varprojlim \left \langle [-1,1],f_n \right \rangle$ there would exist an embedding of $X$ into $\mathbb{R}^2$ for which the point $x \in X$ is accessible (thus giving a positive partial answer to the Nadler-Quinn problem, for the special case of arc-like continua which are inverse limits of simplicial inverse systems of arcs).  This is because for a simplicial system, for any $n$, there are only finitely many possible radial contour factors for the maps $f_n$, $f_n \circ f_{n+1}$, $f_n \circ f_{n+1} \circ f_{n+2}$, $\ldots$, hence by replacing the bonding maps with appropriate compositions we may freely assume without loss of generality that $f_n$ and $f_n \circ f_{n+1}$ have the same radial contour factor for each $n$.

\bibliographystyle{amsplain}
\bibliography{Departures}

\providecommand{\bysame}{\leavevmode\hbox to3em{\hrulefill}\thinspace}
\providecommand{\MR}{\relax\ifhmode\unskip\space\fi MR }
\providecommand{\MRhref}[2]{%
  \href{http://www.ams.org/mathscinet-getitem?mr=#1}{#2}
}
\providecommand{\href}[2]{#2}
\begin{thebibliography}{10}

\bibitem{anderson-choquet1959}
R.~D. Anderson and Gustave Choquet, \emph{A plane continuum no two of whose
  nondegenerate subcontinua are homeomorphic: {A}n application of inverse
  limits}, Proc. Amer. Math. Soc. \textbf{10} (1959), 347--353. \MR{105073}

\bibitem{anusic2021}
Ana Anu\v{s}i\'{c}, \emph{Planar embeddings of {M}inc's continuum and
  generalizations}, Topology Appl. \textbf{292} (2021), Paper No. 107519, 13.
  \MR{4223408}

\bibitem{anusic-bruin-cinc2017}
Ana Anu\v{s}i\'{c}, Henk Bruin, and Jernej \v{C}in\v{c}, \emph{Uncountably many
  planar embeddings of unimodal inverse limit spaces}, Discrete Contin. Dyn.
  Syst. \textbf{37} (2017), no.~5, 2285--2300. \MR{3619063}

\bibitem{anusic-bruin-cinc2020}
\bysame, \emph{Planar embeddings of chainable continua}, Topology Proc.
  \textbf{56} (2020), 263--296. \MR{4061353}

\bibitem{brechner1978}
Beverly Brechner, \emph{On stable homeomorphisms and imbeddings of the pseudo
  arc}, Illinois J. Math. \textbf{22} (1978), no.~4, 630--661. \MR{500860}

\bibitem{brown1960}
Morton Brown, \emph{Some applications of an approximation theorem for inverse
  limits}, Proc. Amer. Math. Soc. \textbf{11} (1960), 478--483. \MR{115157}

\bibitem{debski-tymchatyn1993}
W.~D\polhk{e}bski and E.~D. Tymchatyn, \emph{A note on accessible composants in
  {K}naster continua}, Houston J. Math. \textbf{19} (1993), no.~3, 435--442.
  \MR{1242430}

\bibitem{problems2018}
Logan~C. Hoehn, Piotr Minc, and Murat Tuncali, \emph{Problems in continuum
  theory in memory of {S}am {B}. {N}adler, {J}r.}, Topology Proc. \textbf{52}
  (2018), 281--308, With contributions. \MR{3773586}

\bibitem{problems2002}
Alejandro Illanes, Sergio Mac\'{i}as, and Wayne Lewis (eds.), \emph{Continuum
  theory}, Lecture Notes in Pure and Applied Mathematics, vol. 230, Marcel
  Dekker, Inc., New York, 2002. \MR{2001430}

\bibitem{lewis1981}
Wayne Lewis, \emph{Embeddings of the pseudo-arc in {$E\sp{2}$}}, Pacific J.
  Math. \textbf{93} (1981), no.~1, 115--120. \MR{621602}

\bibitem{lewis1983}
\bysame, \emph{Continuum theory problems}, vol.~8, 1983, pp.~361--394.
  \MR{765091}

\bibitem{mayer1982}
John~C. Mayer, \emph{Embeddings of plane continua and the fixed point
  property}, ProQuest LLC, Ann Arbor, MI, 1982, Thesis (Ph.D.)--University of
  Florida. \MR{2632164}

\bibitem{mayer1983}
\bysame, \emph{Inequivalent embeddings and prime ends}, vol.~8, 1983,
  pp.~99--159. \MR{738473}

\bibitem{mazurkiewicz1929}
S.~Mazurkiewicz, \emph{Un th\'eor\`eme sur l'accessibilit\'e des continus
  ind\'ecomposables}, Fund. Math. \textbf{14} (1929), 271--276.

\bibitem{minc1997}
Piotr Minc, \emph{Embedding of simplicial arcs into the plane}, vol.~22, 1997,
  pp.~305--340. \MR{1718899}

\bibitem{minc-transue1992}
Piotr Minc and W.~R.~R. Transue, \emph{Accessible points of hereditarily
  decomposable chainable continua}, Trans. Amer. Math. Soc. \textbf{332}
  (1992), no.~2, 711--727. \MR{1073777}

\bibitem{nadler1972}
Sam~B. Nadler, Jr., \emph{Some results and problems about embedding certain
  compactifications}, Proceedings of the {U}niversity of {O}klahoma {T}opology
  {C}onference {D}edicated to {R}obert {L}ee {M}oore ({N}orman, {O}kla., 1972),
  Univ. of Oklahoma, Norman, Okla., 1972, pp.~222--233. \MR{0362256}

\bibitem{nadler-quinn1972}
Sam~B. Nadler, Jr. and J.~Quinn, \emph{Embeddability and structure properties
  of real curves}, Memoirs of the American Mathematical Society, No. 125,
  American Mathematical Society, Providence, R.I., 1972. \MR{0353278}

\bibitem{ozbolt2020}
Joseph~Stephen Ozbolt, \emph{On {P}lanar {E}mbeddings of the {K}naster
  {$V\Lambda$}-{C}ontinuum}, ProQuest LLC, Ann Arbor, MI, 2020, Thesis
  (Ph.D.)--Auburn University. \MR{4578910}

\end{thebibliography}

\end{document}